\documentclass[10pt]{amsart}

\usepackage{amsthm,amscd,amssymb,amsmath,tikz-cd,enumerate}
\usepackage[mathscr]{eucal}

\theoremstyle{definition}

\theoremstyle{plain}
\newtheorem{defn}{Definition}[section]

\newtheorem{lemma}[defn]{Lemma}
\newtheorem{theorem}[defn]{Theorem}
\newtheorem{proposition}[defn]{Proposition}
\newtheorem{corollary}[defn]{Corollary}
\newtheorem{conjecture}[defn]{Conjecture}

\theoremstyle{remark}
\newtheorem{remark}[defn]{Remark}

\newcommand{\HH}{{H}}
\newcommand{\rank}{{\mathrm {rank} \,}}
\newcommand{\Pic}{{\text {Pic}}}
\newcommand{\Hom}{{\mathrm {Hom}}}
\newcommand{\Gr}{\mathrm{Gr}}

\usepackage[margin=1in]{geometry}
\begin{document}
\title{{ Hodge numbers of Landau--Ginzburg models}}

\author{Andrew Harder}
\address{Department of Mathematics, Lehigh University, Bethlehem, Pennsylvania, USA 18015}
\email[A.~Harder]{anh318@lehigh.edu}

\keywords{Algebraic geometry, mirror symmetry, Hodge theory, toric varieties}

\begin{abstract}
We study the Hodge numbers $f^{p,q}$ of Landau--Ginzburg models as defined by Katzarkov, Kontsevich, and Pantev. First we show that these numbers can be computed using ordinary mixed Hodge theory, then we give a concrete recipe for computing these numbers for the Landau--Ginzburg mirrors of Fano threefolds. We finish by proving that for a crepant resolution of a Gorenstein toric Fano threefold $X$ there is a natural LG mirror $(Y,w)$ so that $h^{p,q}(X) = f^{3-q,p}(Y,w)$.
\end{abstract}
\maketitle

\section{Introduction}

The goal of this paper is to study Hodge theoretic invariants associated to the class of Landau--Ginzburg models which appear as the mirrors of Fano varieties in mirror symmetry.

Mirror symmetry is a phenomenon that arose in theoretical physics in the late 1980s. It says that to a given Calabi--Yau variety $W$ there should be a dual Calabi--Yau variety $W^\vee$ so that the A-model TQFT on $W$ is equivalent to the B-model TQFT on $W^\vee$ and vice versa. The A- and B-model TQFTs associated to a Calabi--Yau variety are built up from symplectic and algebraic data respectively. Consequently the symplectic geometry of $W$ should be related to the algebraic geometry of $W^\vee$ and vice versa. A number of precise and interrelated mathematical approaches to mirror symmetry have been studied intensely over the last several decades. Notable approaches to studying mirror symmetry include homological mirror symmetry \cite{kont}, SYZ mirror symmetry \cite{syz}, and the more classical enumerative mirror symmetry. A basic expectation common to most (perhaps all) forms of mirror symmetry is ``topological mirror symmetry''. To a projective algebraic variety $X$, we may associate its Hodge numbers, $h^{p,q}(X) := \dim H^q(X,\Omega^p_X)$. We say that  two $d$-dimensional projective Calabi--Yau manifolds $W$ and $W^\vee$ are {\em topologically mirror dual} if 
\begin{equation}\label{hnms}
h^{p,q}(W) = h^{d-q,p}(W^\vee)
\end{equation}
for all $p$ and $q$. One of the major problems that is discussed in the mirror symmetry literature is the question of how, given a Calabi--Yau manifold $W$, one produces a variety $W^\vee$ that is mirror dual of $W$ \cite{bat,bor,gs,dht}. These constructions are often synthetic and combinatorial, and they rarely address the question of whether $W$ and $W^\vee$ are in fact mirror dual in any sophisticated sense. One of the major advantages of topological mirror symmetry is that it provides a simple test of whether a given construction produces a mirror object or not.

More recently, there has been significant interest in extensions of mirror symmetry for manifolds which are not Calabi--Yau. An example of this, which goes back at least to work of Givental \cite{giv} in the mathematical literature, is mirror symmetry for Fano manifolds. While mirror symmetry for Calabi--Yau manifolds predicts that Calabi--Yau manifolds appear in mirror pairs, mirror symmetry for Fano manifolds predicts that there is a mirror relationship between Fano manifolds and objects called Landau--Ginzburg models. A {\em Landau--Ginzburg model} (abbreviated as {\em LG model} from this point on) is, in its broadest sense, simply a complex manifold $Y$ equipped with a holomorphic map $w$ from $Y$ to $\mathbb{C}$, and a complexified K\"ahler form $\omega_Y$. Again, there are several precise formulations of mirror symmetry for Fano manifolds and LG models which follow roughly the same lines as mirror symmetry for projective Calabi--Yau manifolds. For instance, Auroux, Katzarkov, and Orlov have studied homological mirror symmetry for del Pezzo surfaces, \cite{ako2}, and Auroux has studied SYZ mirror symmetry \cite{aur1} for projective varieties with effective anticanonical divisor. 

A natural question to ask is: if $X$ is a Fano manifold which is mirror to a LG model $(Y,w)$, is there an analogue of (\ref{hnms}) relating Hodge theoretic data on $X$ to Hodge theoretic data on $(Y,w)$? 

To pose this question properly, we should start by being more precise about the Hodge theoretic data that should appear in our putative analogue of (\ref{hnms}). Since $X$ is a projective manifold, it is natural to take the Hodge numbers $h^{p,q}(X)$ to be its corresponding Hodge theoretic data. 

The appropriate Hodge-theoretical data to associate to a LG model is a little less clear. There is a long history of studying the Hodge theory of pairs $(Y,w)$, but usually this data is in the form of a bundle on $\mathbb{C}$ with connection, along with certain decorations (e.g. Brieskorn lattices \cite{odl}, TERP structures \cite{hert}, non-commutative Hodge structures \cite{kkp1}). At first glance, it is unclear how to extract clean numerical invariants from these structures that could reflect the Hodge numbers of $X$, so in \cite{kkp2} Katzarkov, Kontsevich and Pantev give a definition of invariants of LG models which they expect to play the role of Hodge numbers. We will call these invariants the {\em KKP Hodge numbers} of $(Y,w)$ and write them as $f^{p,q}(Y,w)$ (see Definition \ref{def:kkphn} for details). Their relationship with the Hodge theoretic invariants of LG models mentioned above was explained in work of Esnault, Sabbah, and Yu \cite{esy}, as well as in work of Shamoto \cite{sham}.

Katzarkov, Kontsevich, and Pantev \cite{kkp2} argue that if $X$ is a Fano manifold of dimension $d$ and if a LG model $(Y,w)$ is homologically mirror to $X$, then a version of (\ref{hnms}) should hold between $X$ and $(Y,w)$. Precisely, we expect that
\begin{equation}\label{tmdfano}
h^{p,q}(X) = f^{d-p,q}(Y,w).
\end{equation}
If a pair composed of a Fano variety $X$ and an LG model $(Y,w)$ satisfy (\ref{tmdfano}), we will say that they are {\em topologically mirror dual}. Our goal in this article is to study KKP Hodge numbers and explain how they can be computed in concrete examples. In \cite{lp}, Lunts and Przyjalkowski show that del Pezzo surfaces are topologically mirror dual to their homological mirror duals, but to our knowledge, this paper contains the first complete computation of KKP Hodge numbers in dimension greater than 2. In a different framework, related computations have been done by Gross, Katzarkov, and Ruddat \cite{gkr}, relating the Hodge numbers of hypersurfaces in toric varieties (with no restrictions on Kodaira dimension) to the Hodge numbers of mirror Landau--Ginzburg models. Their setup uses a different approach to the Hodge theory of Landau--Ginzburg models, and the Landau-Ginzburg mirrors that they analyze have dimension greater than that of the original variety. The relationship between their setup and ours is discussed in the case of the cubic threefold in $\mathbb{P}^4$ in an unpublished preprint version of \cite{gkr}. The reader may consult \cite[Section 7]{gkr2} for details.

We will focus on the case where the map $w$ is proper, and we obtain our strongest results in the case where $\dim Y = 3$ and $Y$ and admits a compactification of a certain type (see Definition \ref{defn:wf} for details). As we will explain in Section \ref{sect:expectations}, any Fano threefold is expected to have mirror LG model of this type, so this provides a complete toolbox for studying the KKP Hodge numbers of mirrors to Fano threefolds. We exhibit this by computing the KKP Hodge numbers of a class of LG models which are mirror to certain toric threefolds. The results in this paper have recently been used by Cheltsov and Przyjalkowski to compute the KKP Hodge numbers of a more complicated class of 3-dimensional LG models \cite{cp}.

\subsection*{Outline} This paper is organized as follows. Section \ref{sect:over} contains most of the necessary background for the rest of the paper. We will begin by reviewing several facts about mixed Hodge structures that will be applied later on. We will then briefly discuss mirror symmetry for Fano manifolds and LG models, solidify our notation regarding LG models, and define KKP Hodge numbers.

Section \ref{sect:hdg} is dedicated to proving some general preliminary results about KKP Hodge numbers. We first show that the KKP Hodge numbers can be computed using classical mixed Hodge theory. In \cite{kkp2}, it is shown that $\sum_{p+q = i} f^{p,q}(Y,w) = h^i(Y,V)$, where $h^{i}(Y,V)$ denotes the dimension of the $i^\mathrm{th}$ relative cohomology of $Y$ with respect to a generic smooth fiber $V$ of $w$. The group $H^i(Y,V)$ itself bears a mixed Hodge structure, hence a Hodge filtration. We will show in Theorem \ref{thm:filt} that the dimensions of the graded pieces of this Hodge filtration agree with $f^{p,q}(Y,w)$. This result essentially follows from a careful reading of work of Katzarkov, Kontsevich, and Pantev \cite[Claim 2.22]{kkp2}. We will then prove two theorems that are likely known to experts, but do not seem to appear in the literature anywhere. First, in Theorem \ref{thm:poincare}, we will prove that if $(Y,w)$ is a proper LG model, $V$ is a smooth fiber of $w$ and $\dim Y = d$, then $h^i(Y,V) = h^{2d-i}(Y,V)$ for all $i$. Second, we will show how to compute $h^2(Y,V)$ in the case where $w$ is a proper map. Let $\Sigma$ be the set of critical values of $w$ and let $\rho_s$ denote the number of irreducible components in $w^{-1}(s)$. Define 
\[
k(Y,w) := \sum_{s\in \Sigma}(\rho_s-1).
\]
In Theorem \ref{prop:folk}, we will show that if $h^{1}(V) = 0$ and $w$ is proper, then $h^2(Y,V) = k(Y,w)$.

Section \ref{ex:three} gives a complete computation of the KKP Hodge numbers of LG models in dimension 3 which satisfy certain natural conditions. In examples (see \cite{prz}), the mirror to a Fano threefold is a LG model $(Y,w)$ which has several restrictive properties. First, $Y$ admits a smooth projective compactification $Z$ so that $w :Y \rightarrow \mathbb{C}$ extends to a projective map $f: Z \rightarrow \mathbb{P}^1$. The following conditions also hold.
\begin{enumerate}
\item A general smooth fiber $V$ of $f$ is a K3 surface which is anticanonical in $Z$.
\item The fiber of $f$ over $\infty$ is a simple normal crossings union of rational surfaces whose dual intersection complex is a triangulation of a sphere.
\item  $h^{i,0}(Z) = 0$ for $i \neq 0$. 
\end{enumerate}
The data of $(Z,f)$ satisfying these conditions will be called a {\em type III compactification of $(Y,w)$}. Let $ph(Y,w)$ be the dimension of the cokernel of $H^2(Y;\mathbb{Q}) \rightarrow H^2(V;\mathbb{Q})$.  The main theorem of Section \ref{ex:three} says that the KKP Hodge numbers of a LG model with a type III compactification are expressed in terms of $k(Y,w)$, $ph(Y,w)$ and $h^{2,1}(Z)$.
\begin{theorem}[Theorem \ref{thm:3fold}]\label{thm:3foldintro}
If $(Y,w)$ admits a type III compactification $(Z,f)$, then
\begin{align*}
f^{3,0}(Y,w) & = f^{0,3}(Y,w) = 1,\\ 
f^{1,1}(Y,w) &= f^{2,2}(Y,2) = k(Y,w),\\  
f^{2,1}(Y,w) &= f^{1,2}(Y,w) = ph(Y,w) -2 + h^{2,1}(Z),
\end{align*}
and $f^{p,q}(Y,w)=0$ for all other values of $p,q$.
\end{theorem}
We will conclude in Section \ref{sect:toric} by showing that, in dimension 3, topological mirror symmetry holds for crepant resolutions of Gorenstein toric Fano varieties. For each weak Fano toric threefold $\hat{X}_\Delta$, we will construct an LG model $(Y_\Delta, w_\Delta)$ and a type III compactification of $(Y_\Delta,w_\Delta)$. We conjecture that these LG models are mirror to $\hat{X}_\Delta$. We will then use Theorem \ref{thm:3fold} to compute the KKP Hodge numbers of $(Y_\Delta, w_\Delta)$ explicitly and show in Theorem \ref{thm:weakfanohn} that 
\[
h^{p,q}(\hat{X}_\Delta) = f^{3-p,q}(Y_\Delta, w_\Delta) 
\]
for all $p,q$. Hence $(Y_\Delta,w_\Delta)$ and $\hat{X}_\Delta$ are topologically mirror to one another. The proof of this theorem relies on basic combinatorial dualities relating counts of integral points in reflexive 3-dimensional polytopes and their polar dual polytopes.
\subsection*{Conventions}
Throughout this paper, we will often refer to cohomology groups without specifying coefficients (e.g. $H^i(X)$). In these cases, the reader may interpret $H^i(X)$ to mean either cohomology with complex or rational coefficients. When discussing cohomology groups admitting Hodge structures, we will usually write cohomology with rational coefficients unless otherwise specified.

All algebraic varieties in this paper will be assumed to be over $\mathbb{C}$, and we consider them with respect to their analytic topology.

\subsection*{Acknowledgements}
I would like to thank Valery Lunts for many valuable suggestions and comments. I would also like to thank Charles Doran, Ludmil Katzarkov, and Victor Przyjalkowski for useful conversations. I was partially supported by an NSERC postgraduate scholarship and the Simons Collaboration in Homological Mirror Symmetry during the preparation of this paper.

\section{Background}\label{sect:over}

Here we will review basic facts about mixed Hodge structures which we will find useful, along with a brief discussion of mirror symmetry for Fano manifolds and LG models.

\subsection{Mixed Hodge structures}
This section is devoted to reviewing some basic facts about mixed Hodge structures. The book of Peters and Steenbrink \cite{pet-st} contains proofs of most of the facts mentioned here. We will assume that the reader has a basic understanding of pure Hodge structures.

\begin{defn}
Let $R$ be a finite dimensional vector space over $\mathbb{Q}$. Assume that we have two filtrations given as follows.
\begin{enumerate}[--]
\item An ascending filtration $W_\bullet$ on $R$ called the {\em weight filtration}. 
\item A descending filtration $F^\bullet$ on $R_\mathbb{C}:=R \otimes \mathbb{C}$ called the {\em Hodge filtration}. 
\end{enumerate}
Let $W^\mathbb{C}_\bullet$ be the $\mathbb{C}$-linear extension of $W_\bullet$ to $R_\mathbb{C}$. Let
$$
\mathrm{Gr}^W_iR = W_i/W_{i-1}, \qquad \mathrm{Gr}_F^jR = F^j/F^{j+1}.
$$
The data $(R,F,W)$ is a {\em mixed Hodge structure} if on the $i^\mathrm{th}$ graded component of the weight filtration, $\mathrm{Gr}^W_iV$, the induced Hodge filtration
$$
F^j \mathrm{Gr}^{W^\mathbb{C}}_iR_\mathbb{C} = \mathrm{Im}( F^j \cap W^\mathbb{C}_i \rightarrow \mathrm{Gr}^{W^\mathbb{C}}_iR_\mathbb{C})
$$
defines a pure Hodge structure of weight $i$.
\end{defn}
\begin{defn}
Let $(R,F,W)$ be a mixed Hodge structure. We define $i^{p,q}(R)$ to be the dimension of $\mathrm{Gr}_F^p \mathrm{Gr}^{W^\mathbb{C}}_{p+q}R_\mathbb{C}$.
\end{defn}
From this definition it follows that $\dim \mathrm{Gr}_F^pR_\mathbb{C} = \sum_{q} i^{p,q}(R)$. According to \cite[Corollary 3.8]{pet-st}, if there is an exact sequence of mixed Hodge structures
\begin{equation}\label{eshodge}
R' \longrightarrow R \longrightarrow R''
\end{equation}
then for each $p,q$ we get an exact sequence
$$
\mathrm{Gr}_F^p\mathrm{Gr}^{W^\mathbb{C}}_{p+q} R'_\mathbb{C}  \longrightarrow \mathrm{Gr}_F^p\mathrm{Gr}^{W^\mathbb{C}}_{p+q} R_\mathbb{C} \longrightarrow \mathrm{Gr}_F^p\mathrm{Gr}^{W^\mathbb{C}}_{p+q} R_\mathbb{C}''.
$$
Moreover, if (\ref{eshodge}) is a short exact sequence of mixed Hodge structures, then Hodge numbers are additive. In other words,
$$
i^{p,q}(R') + i^{p,q}(R'') = i^{p,q}(R).
$$
If $(R,F,W)$ is a mixed Hodge structure, then $R^\vee = \Hom_\mathbb{Q}(R,\mathbb{Q})$ is naturally equipped with mixed Hodge structure with $i^{p,q}(R) = i^{-p,-q}(R^\vee)$. 

The {\em Tate Hodge structure}, denoted $\mathbb{Q}(-1)$, is the unique pure Hodge structure of weight $2$ so that $i^{1,1}(\mathbb{Q}(-1)) = 1$ and so that $i^{p,q}(\mathbb{Q}(-1)) = 0$ if $(p,q) \neq (1,1)$. If $R$ is a mixed Hodge structure, then $R(-d)$ denotes $R \otimes (\mathbb{Q}(-1))^{\otimes d}$.

Next, let us discuss variations of Hodge structure and limit mixed Hodge structures.
\begin{defn}
Let $U$ be a complex manifold. A {\em variation of Hodge structure over $U$ of weight $n$} is the following data.
\begin{enumerate}[--]
\item A vector bundle $\mathscr{R}$ on $U$.
\item A flat connection $\nabla: \mathscr{R} \rightarrow \Omega^1_U \otimes_{\mathscr{O}_U} \mathscr{R}$.
\item A rational local system $\mathsf{R}$ so that $\mathsf{R} \otimes \mathbb{C} \cong \ker \nabla$.
\item A finite descending filtration $\mathscr{F}^n \subseteq \mathscr{F}^{n-1} \subseteq \dots \subseteq \mathscr{F}^0 = \mathscr{R}$ by sub-bundles.
\end{enumerate}
We require that at each point $t \in U$, the filtration $\mathscr{F}^n_t \subseteq \dots \subseteq \mathscr{F}^0_t$ defines a pure Hodge structure of weight $n$ on $\mathsf{R}_t$, and that $\nabla \mathscr{F}^j \subseteq \mathscr{F}^{j-1}$ for each $j$.  
\end{defn}
Suppose we have a variation of Hodge structure $(\mathscr{R}, \mathscr{F}^\bullet,\nabla)$ of weight $n$ over the punctured disc $\mathbf{D}^* = \mathbf{D} \setminus \{0\}$. If $t \in \mathbf{D}^*$ and $\gamma$ is a generator of $\pi_1(\mathbf{D}^*, t)$, associated to a counterclockwise loop, then $\gamma$ acts on $\mathsf{R}_t$ by parallel transport. We call this the {\em monodromy action on $\mathsf{R}_t$} and we will denote it $T$.

According to work of Schmid \cite{sch}, one may associate a mixed Hodge structure to the data of a variation of Hodge structure of weight $n$ over $\mathbf{D}^*$ whose underlying vector space is $\mathsf{R}_t$. We use the notation $(R_\mathrm{lim},F_\mathrm{lim},W_\mathrm{lim})$ to denote this mixed Hodge structure.  The Hodge filtration, whose precise definition we will not review here, is determined by the asymptotic behaviour of the Hodge filtration on $\mathbf{D}^*$ and has the useful property that $\dim \mathrm{Gr}_{F_\mathrm{lim}}^j R_\mathrm{lim} = \dim \mathrm{Gr}_{\mathscr{F}_t}^j \mathscr{R}_t$ for all $j$. The weight filtration is determined by $N := \log T$, and is given by the formula
\[
W_{\ell+n} = \sum_{k-j=\ell}\ker N^{k+1} \cap \mathrm{im}\, N^j.
\]
The operator $N$ acts on $R_\mathrm{lim}$ in a very nice way. Particularly, we have that $N$ induces a morphism of mixed Hodge structures from $R_\mathrm{lim}$ to $R_\mathrm{lim}(-1)$ and hence a morphism pure Hodge structures from $\mathrm{Gr}^W_iR_\mathrm{lim}$ to $\mathrm{Gr}^W_{i-2}R_\mathrm{lim}(-1)$. Moreover, $N^k$ induces an isomorphism of pure Hodge structures between $\mathrm{Gr}^W_{n+k}R_\mathrm{lim}$ and $\mathrm{Gr}^W_{n-k}R_\mathrm{lim}(-k)$. Therefore if $i > n$, $N: \mathrm{Gr}^W_iR_\mathrm{lim} \rightarrow \mathrm{Gr}^W_{i-2}R_\mathrm{lim}(-1)$ is injective and if $i \leq n$, the same map is surjective.

%---------------------------------------------------------------------------------------------------------------------------

\subsection{Mixed Hodge structure on the cohomology of varieties}
We will now discuss some ways in which mixed Hodge structures appear in algebraic geometry. Deligne \cite{del-th1,del-th2} has proven that if $U$ is an algebraic variety over $\mathbb{C}$ then $H^i(U;\mathbb{Q})$ admits a canonical mixed Hodge structure. If $U$ is smooth and projective, then $H^d(U;\mathbb{Q})$ carries a pure Hodge structure of weight $d$. There are two other cases that we will find useful to discuss in this paper: the case where $U$ is smooth and quasiprojective but not necessarily projective, and the case where $U$ is projective and has simple normal crossings.

A projective variety $U$ of dimension $n$ has {\em simple normal crossings} if for every $p \in U$, there is an open analytic neighbourhood $U_p$ of $p$ which is analytically isomorphic to a neighbourhood of 0 in the set $\{ x_1x_2\dots x_k = 0\} \subset \mathbb{C}^{n+1}$ for some $k$. We can compute the cohomology of $U$, and even the weight graded pieces of its mixed Hodge structure using the following result.
\begin{theorem}[{\cite[pp. 103]{mor}, \cite[Section 4]{cs}}]\label{thm:mv}
Let $U$ be a projective $d$-dimensional normal crossings variety with irreducible components $U_1,\dots, U_k$. Let
$$
U^{[n]} = \bigsqcup_{\substack{I \subset [1,k]\\ |I| =n+1 }} \cap_{i\in I} U_i
$$
be the disjoint union of all codimension $n$ strata of $U$. There is a spectral sequence with $E_1$ term
$$
E_1^{p,q} = H^{q}(U^{[p]};\mathbb{Q})
$$
which degenerates to $H^{p+q}(U;\mathbb{Q})$ at the $E_2$ term. There is an isomorphism of mixed Hodge structures between $E_2^{p,q}$ and $\mathrm{Gr}_W^qH^{p+q}(U;\mathbb{Q})$.
\end{theorem}
\begin{corollary}\label{cor:ncomp}
If $U$ is a normal crossings variety of dimension $d$, and $n: \widetilde{U} \rightarrow U$ is its normalization, then 
\[
n^* : H^{2d}(U;\mathbb{Q}) \rightarrow H^{2d}(\widetilde{U};\mathbb{Q})
\]
is an isomorphism. Hence $\dim H^{2d}(U;\mathbb{Q})$ is equal to the number of irreducible components of $U$.
\end{corollary}

If $U$ is a simple normal crossings variety of dimension $d$, then there is an associated simplicial complex called the {\em dual intersection complex} of $U$, denoted $\Gamma_U$. The underlying topological space of $\Gamma_U$ is called $|\Gamma_U|$. If $U_I$ denotes $\cap_{i \in I}$ for  $I \subset [1,k]$, then each $U_I$ corresponds to a $d- |I|$ simplex $\sigma_I$ of $\Gamma_U$ and $\sigma_I$ is included in $\sigma_{I'}$ if and only if $U_{I'} \subset U_I$. Another consequence of  Theorem \ref{thm:mv} is the following statement.
\begin{corollary}[{\cite[pp. 105]{mor}}]\label{cor:mor}
If $X$ is a simple normal crossings variety, then
\[
H^i(|\Gamma_U|;\mathbb{Q}) \cong W_0 H^i(U;\mathbb{Q}).
\]
\end{corollary}
We will now describe an approach to computing the mixed Hodge structure on a quasiprojective variety that will be used several times in this paper.

Recall that for any closed subspace $N$ of a manifold $M$, we have a long exact sequence of compactly supported cohomology groups.
\begin{equation}\label{eq:cpsup}
\dots \longrightarrow H^i_c (M\setminus N ;\mathbb{Q}) \longrightarrow H^i_c(M;\mathbb{Q}) \longrightarrow H^i_c(N;\mathbb{Q}) \longrightarrow \dots.
\end{equation}
We also recall that if $M$ is an oriented manifold without boundary, then Poincar\'e duality gives an isomorphism between $H^i_c(M;\mathbb{Q})$ and $H^{\dim_\mathbb{R} M-i}(M;\mathbb{Q})^\vee$. If $M$ and $N$ are smooth algebraic varieties, and $N \subset M$ is an algebraic subvariety, then (\ref{eq:cpsup}) becomes a long exact sequence of mixed Hodge structures. Furthermore, if $M$ is a smooth variety of (complex) dimension $d$, then $H^i_c(M;\mathbb{Q})$ is isomorphic to $H^{2d-i}(M;\mathbb{Q})^\vee(-d)$ as a mixed Hodge structure (see \cite[Theorem 1.7.1]{fujiki} for details). 

Therefore, if $U$ is smooth and quasiprojective and we have a smooth projective compactification $\overline{U}$ of $U$ so that $\overline{U} \setminus U \subset \overline{U}$ is a simple normal crossings divisor, these facts allow us to  compute the mixed Hodge structure on $H^i(U;\mathbb{Q})$. First we compute the mixed Hodge structure on $\overline{U} \setminus U$ using Theorem \ref{thm:mv}, then we apply (\ref{eq:cpsup}).

%--------------------------------------------------------------------------------------------------------------------------------

\subsection{Limit mixed Hodge structures and degenerations}

We now assume we are in the following geometric situation. Let $\pi : \mathscr{U} \rightarrow \mathbf{D}$ be a projective map and assume that $\mathscr{U}$ is a holomorphic manifold. We assume that the fibers over $p \in \mathbf{D}^*$ are smooth. If $\pi^0$ is the restriction of $\pi$ to the preimage of $\mathbf{D}^*$, the local system $R^i \pi^0_*\mathbb{Q}$ carries a variation of Hodge structure, hence there is an associated limit mixed Hodge structure at $0$ whose underlying vector space is isomorphic to $(R\pi_*^0\mathbb{Q})_t \cong H^i(\mathscr{U}_t;\mathbb{Q})$.

\begin{theorem}[{\cite[Proposition C.11]{pet-st}}\label{thm:clcon}]\label{thm:clcon}
Under the assumptions above, there is a strong deformation retract from $\mathscr{U}$ to $\pi^{-1}(0)$.
\end{theorem}
We will now introduce the {\em Clemens--Schmid exact sequence}, which relates the cohomology of $\mathscr{U}$, the cohomology of a smooth fiber of $\pi$, and the monodromy action on $H^i(\mathscr{U}_t,\mathbb{Q})$. In the statement of Theorem \ref{thm:csex}, we use Theorem \ref{thm:clcon} to equip $H^i(\mathscr{U};\mathbb{Q})$ with a mixed Hodge structure.

\begin{theorem} [{\cite[Corollary 11.44]{pet-st}, \cite{clem}}]\label{thm:csex}
Let $\mathscr{U}$ be a K\"ahler manifold and assume that $\pi : \mathscr{U} \rightarrow \mathbf{D}$ is a projective map which is smooth away from $0$.  Let $V$ be a smooth fiber of $\pi$ and let $T_i$ be the action on $H^i(V;\mathbb{Q})$ of monodromy around a counterclockwise loop in $\mathbf{D}^*$. Let $N_i = \log\, T_i$ and let $\mathscr{U}^* = \mathscr{U} \setminus \pi^{-1}(0)$. Then there is a long exact sequence of mixed Hodge structures
\[
\dots \longrightarrow H^i(\mathscr{U},\mathscr{U}^*;\mathbb{Q}) \xrightarrow{\,q_i\,} H^i(\mathscr{U};\mathbb{Q}) \xrightarrow{\,r_i\,} H^i_\mathrm{lim}(V;\mathbb{Q}) \xrightarrow{\,N_i\,} H^i_\mathrm{lim}(V;\mathbb{Q}) \longrightarrow \dots.
\]
On the level of vector spaces, $r_i$ can be identified with the pullback associated to the embedding $r : V \hookrightarrow \mathscr{U}$ and $H^i_\mathrm{lim}(V;\mathbb{Q})$ denotes the limit mixed Hodge structure on $H^i(V;\mathbb{Q})$ at $0$. 
\end{theorem}
Finally, before moving on, we record the global invariant cycles theorem. If $f:U\rightarrow C$ is a proper, dominant morphism of smooth quasiprojective varieties, then there is a Zariski open subset $C^\circ$ of $C$ so that $f$ is smooth on $f^{-1}(C^\circ)$. For every $s\in C^\circ$, there is a monodromy representation
\[
\gamma : \pi_1(C^\circ,s)  \longrightarrow \mathrm{GL}(H^i(U_s;\mathbb{Q}))
\]
where $U_s$ denotes the fiber of $f$ over $s$. 
\begin{theorem}[{\cite[Theorem 4.24, Corollary 4.25]{vois3}, \cite[Th\'eor\`eme 4.1.1]{del-th1}}]\label{thm:gict}
Let $f:U \rightarrow C$ be a proper, dominant morphism of smooth quasiprojective varieties. The restriction map $H^i(U;\mathbb{Q})\rightarrow H^i(U_s;\mathbb{Q})$ is a morphism of mixed Hodge structures whose image is $H^i(U_s;\mathbb{Q})^\gamma$.
\end{theorem}

\subsection{Mirror symmetry and LG models}

The following is a somewhat loose explanation of how the statement that a Fano manifold $X$ is mirror to a LG model $(Y,w)$ should be interpreted. A more precise and thorough discussion may be found in \cite[Section 2]{kkp2}.

A triple $(X,\omega_X,s)$ composed of 
\begin{enumerate}[--]
\item a Fano manifold $X$,
\item a complexified K\"ahler form $\omega_X$ on $X$,
\item a section $s \in H^0(X; K_X^{-1})$ whose vanishing locus has normal crossings,
\end{enumerate}
and a quadruple $((Y,w),\omega_Y, \mathrm{vol}_Y)$ composed of
\begin{enumerate}[--]
\item a quasiprojective manifold $Y$,
\item a surjective regular function $w: Y \rightarrow \mathbb{C}$ with compact critical locus,
\item a complexified K\"ahler form $\omega_Y$,
\item a trivialization $\mathrm{vol}_Y$ of $K_Y^{-1}$,
\end{enumerate}
form a {\em homological mirror pair} if the category of A-branes associated to $(X,\omega_X,s)$ (the derived Fukaya category of $(X,\omega_X)$) is equivalent to the category of B-branes associated to $((Y,w),\omega_Y,\mathrm{vol}_Y)$ (the category of matrix factorizations of $(Y,w)$), and the category of B-branes associated to $(X,\omega_X,s)$ (the derived category of coherent sheaves on $X$) is equivalent to the category of A-branes on $((Y,w),\omega_Y,\mathrm{vol}_Y)$ (the derived Fukaya-Seidel category of $(Y,w)$).

According to \cite[Remark 2.1]{kkp2}, if the vanishing locus of $s$ is smooth, then we expect that $w$ is a proper morphism. Our main focus in this paper is mirror symmetry for Fano threefolds. For all Fano threefolds, there exists a smooth anticanonical hypersurface, so we will restrict ourselves to LG models where $w$ is proper.
\begin{defn}\label{defn:mirror}
Let $(Y,w)$ be a pair composed of a smooth quasiprojective variety $Y$ and a regular map $w : Y \rightarrow \mathbb{C}$. {\em A tame compactification} of $(Y,w)$ is a pair $(Z,f)$ consisting of a smooth, projective compactification $Z$ of $Y$ along with a projective morphism $f : Z \rightarrow \mathbb{P}^1$ so that 
\begin{enumerate}[--]
\item $w = f|_Y$, 
\item $Z \setminus Y$ is simple normal crossings, 
\item $f$ has a pole of multiplicity 1 along each component of $f^{-1}(\infty)$, 
\item the critical locus of $f|_{f^{-1}(\mathbb{C})}$ is contained in $Y$.
\end{enumerate}
\end{defn}
\begin{defn}\label{def:LG}
A {\em proper LG model} is a pair $(Y,w)$ with $Y$ a smooth quasiprojective variety and $w: Y\rightarrow \mathbb{C}$ a regular map so that
\begin{enumerate}[--]
\item $w$ is a proper map,
\item $(Y,w)$ admits a tame compactification.
\end{enumerate}
\end{defn}

\begin{remark}
Many of the results that we will prove do not rely on the existence of the trivialization $\mathrm{vol}_Y$ or a choice of $\omega_Y$, so we have omitted them from our definitions for the sake of simplicity.
\end{remark}
%\begin{remark}\label{rmk:cpt}
%It's not difficult to extend the results in this section to the case where fibers are mildly noncompact. If we assume that there exists a compactification $Z$ of $Y$ so that $w$ extends to a function $f$ on $Z$, that $Z \setminus Y = D = D_\infty \cup D_h$ where $D, D_\infty$ and $D_h$ are normal crossings and $f|_{D_h}$ is locally trivial on each component, then one can show that the KKP Hodge numbers (which we will define in the next section) for $Y \cup D_h$ equipped with the restriction of $f$ are the same as those of $(Y,w)$. Therefore, the restriction to relatively compact $Y$ is a relatively mild condition. A proof of this is outlined in \cite[Chapter 2]{harthe}.
%\qed\end{remark}

Let $(Y,w)$ be a proper LG model and let $(Z,f)$ be a tame compactification of $(Y,w)$. Let $D_\infty := f^{-1}(\infty)$, and let $\Omega_Z^\bullet(\log D_\infty)$ be the usual complex of holomorphic differential forms on $Z$ with logarithmic poles at $D_\infty$. 
\begin{defn}
Let $(Y,w)$ be a proper LG model and let $(Z,f)$ be a tame compactification of $(Y,w)$.
The sheaf  of {\em $f$-adapted holomorphic $i$-forms}, which we will denote $\Omega^i_Z(\log D_\infty, f)$, is the subsheaf made up of logarithmic $i$-forms $\omega \in \Omega^i_Z(\log D_\infty)$ so that $\mathrm{d}f \wedge \omega$ has log poles along $D_\infty$. The natural differential coming from the inclusion of $\Omega_Z^i(\log D_\infty,f)$ into $\Omega_Z^i(\log D_\infty)$ defines a differential on $\Omega^\bullet_Z(\log D_\infty, f)$ which turns it into a subcomplex of $\Omega_Z^\bullet(\log D_\infty)$.  
\end{defn}

Katzarkov, Kontsevich, and Pantev prove (\cite[Lemma 2.19]{kkp2}) that the hypercohomology spectral sequence for the stupid filtration on $\Omega_Z^\bullet(\log D_\infty,f)$ degenerates at the $E_1$ term.  Therefore,
$$
\dim \HH^i(Y,w)  = \sum_{p+q=i}\dim \HH^p(Z,\Omega_Z^q(\log D_\infty, f)).
$$
\begin{defn}\label{def:kkphn}
The {\em KKP Hodge numbers of $(Y,w)$} are the invariants 
\[
f^{p,q}(Y,w) := \dim \HH^q(Z,\Omega_Z^p(\log D_\infty, f)).
\]
\end{defn}
If $X$ and $(Y,w)$ form a homological mirror pair, then $f^{p,q}(Y,w)$ are expected to reflect the Hodge numbers of $X$.
\begin{conjecture}[{\cite[Conjecture 3.7]{kkp2}}]\label{conj:KKP}
If $(X,\omega_X,s)$ and $((Y,w),\omega_Y, \mathrm{vol}_Y)$ form a mirror pair and $\dim X = \dim Y = d$, then
\begin{equation}\label{mstop}
h^{p,q}(X) = f^{d-p,q}(Y,w)
\end{equation}
for all $p,q$.
\end{conjecture}
As mentioned in the introduction, if $X$ is a Fano manifold and $(Y,w)$ is a LG model so that $\dim X = \dim Y = d$, and (\ref{mstop}) is satisfied, then we say that $X$ and $(Y,w)$ form a {\em topologically mirror pair}.

\section{KKP Hodge numbers of proper LG models}\label{sect:hdg}

This section is devoted to proving several general results about the KKP Hodge numbers of proper LG models. These are results that will hold in arbitrary dimension. In subsequent sections, we will apply the results in this section to prove results in the special case where $\dim Y = 3$.

\subsection{Relation between filtrations}\label{sect:relfilt}

The goal of this subsection will be to show how one may use the mixed Hodge structure on certain relative cohomology groups to compute $f^{p,q}(Y,w)$ when $(Y,w)$ is a proper LG model. We will let $V = w^{-1}(t)$ for $t$ any regular value of $w$. Then \cite[Lemma 2.21]{kkp2} says that
$$
\dim \HH^i(Y,V) = \dim \HH^i(Y,w).
$$
There is a natural mixed Hodge structure on $H^i(Y,V;\mathbb{Q})$ and Hodge filtration of this mixed Hodge structure can be defined in the following way (see e.g. \cite[pp. 222]{vois3}). Let $\Omega^i_Z(\log D_\infty, \mathrm{rel} \, V)$ be the kernel of the natural restriction map
$$
\iota^*: \Omega_Z^i(\log D_\infty) \longrightarrow \iota_* \Omega^i_V.
$$
where $\iota:V \hookrightarrow Z$ is the natural embedding. There is an isomorphism between $\mathbb{H}^i(Z,\Omega_Z^\bullet(\log D_\infty, \mathrm{rel}\, V))$ and $H^i(Y,V;\mathbb{C})$. The Hodge filtration on $H^i(Y,V;\mathbb{C})$ is defined in the standard way. First define the Hodge filtration on the complex $\Omega^\bullet_Z(\log D_\infty,\mathrm{rel}\, V)$ to be given by the subcomplexes
$$
F^p \Omega^\bullet_{Z}(\log D_\infty,\mathrm{rel}\, V) := \dots \rightarrow 0 \rightarrow \Omega^p_{Z}(\log D_\infty,\mathrm{rel}\,V ) \rightarrow \Omega^{p+1}_{Z}(\log D_\infty,\mathrm{rel}\, V) \rightarrow \dots.
$$
Then we define $F^p\mathbb{H}^i(Z,\Omega_Z^\bullet(\log D_\infty, \mathrm{rel}\, V))$ to be the image of the natural map in hypercohomology
$$
\mathbb{H}^i(Z,F^p\Omega^\bullet_{Z}(\log D_\infty,\mathrm{rel}\, V)) \rightarrow \mathbb{H}^i(Z,\Omega^\bullet_{Z}(\log D_\infty,\mathrm{rel}\, V)).
$$
The spectral sequence associated to this filtration degenerates at the $E_1$ term. Therefore, the dimension of the $p^\mathrm{th}$ Hodge graded piece of $\HH^{p+q}(Y,V;\mathbb{C})$ is 
\[
h^{p,q}(Y,V) := \dim H^q(Z,\Omega_Z^p(\log D_\infty,\mathrm{rel}\, V)).
\]
Combining this with the discussion in the previous section, it follows that
\[
\sum_{p+q = i} h^{p,q}(Y,V) = \dim H^{i}(Y,V) = \sum_{p+q=i}f^{p,q}(Y,w).
\]
We will now show that this equality may be refined to an equality between $h^{p,q}(Y,V)$ and $f^{p,q}(Y,w)$. This will allow us to apply standard techniques in Hodge theory to compute the KKP Hodge numbers of an LG model.
\begin{theorem}\label{thm:filt}
Let $(Y,w)$ be a proper LG model and let $V$ be a smooth fiber of $w$. Then
$$
h^{p,q}(Y,V) = f^{p,q}(Y,w)
$$
for all $p$ and $q$.
\end{theorem}

\begin{proof}
In the proof of \cite[Claim 2.22]{kkp2}, the Katzarkov, Kontsevich, and Pantev construct an object which they call $\mathsf{E}_{\mathscr{Z}/\mathbf{D}}^\bullet$. Let $\mathbf{D}$ be a small disc in $\mathbb{P}^1$ with center at $\infty$ and parameter $\epsilon$. We then let $\mathscr{Z} = Z \times \mathbf{D}$, and we let $p$ be the projection of $\mathscr{Z}$ onto $\mathbf{D}$. Let $\mathscr{D}_\infty$ be the divisor $D_\infty \times \mathbf{D}$ in $\mathscr{Z}$. We then have that $f \times \mathrm{id}$ gives a map from $\mathscr{Z}$ to $\mathbb{P}^1 \times \mathbf{D}$. Let $\Gamma$ be the preimage of the diagonal of $\mathbf{D} \times \mathbf{D} \subset \mathbb{P}^1 \times \mathbf{D}$ under the map under $f \times \mathrm{id}$. Briefly, $\Gamma$ is the divisor in $\mathscr{Z}$ so that under the projection to $Z$, the fiber over $p \in \mathbf{D}$ goes to $f^{-1}(p)$. As usual one lets $\Omega_{\mathscr{Z}/\mathbf{D}}^1(\log \mathscr{D}_\infty)$ be the quotient of $\Omega^1_\mathscr{Z}(\log \mathscr{D}_\infty)$ by $p^{-1}\Omega_\mathbf{D}^1$. One then defines $\Omega_{\mathscr{Z}/\mathbf{D}}^a(\log \mathscr{D}_\infty)$ to be $\bigwedge^a\Omega^1_{\mathscr{Z}/\mathbf{D}}(\log \mathscr{D}_\infty)$. We have that the restriction of this sheaf to $p^{-1}(\epsilon)$ is simply $\Omega_Z(\log D_\infty)$ for any $\epsilon$. Similarly, we have the complex of sheaves $\Omega^a_{\Gamma/\mathbf{D}}(\log D_\Gamma)$ where $D_\Gamma = \mathscr{D}_\infty \cap \Gamma$, however, one must replace $p^{-1}\Omega_\mathbf{D}^1$ with $p^{-1}\Omega_\mathbf{D}^1(\log \infty)$ in the definition given above, since the fiber over $\infty$ is allowed to be singular. Note that $\Gamma$ is simply $f^{-1}(\mathbf{D})$ and $D_\Gamma$ is $f^{-1}(\infty)$. The natural differential then induces a differential on these two complexes, and if we let $i_\Gamma: \Gamma \rightarrow \mathscr{Z}$ be the natural embedding, then we may define
$$
\Omega^\bullet_{\mathscr{Z}/\Delta}(\log \mathscr{D}_\infty,\mathrm{rel}\,f) = \mathrm{ker}(\Omega_{\mathscr{Z}/\Delta}^\bullet(\log \mathscr{D}_\infty) \longrightarrow i_{\Gamma*}\Omega_{\Gamma/\mathbf{D}}^\bullet(\log D_\Gamma)).
$$
The complex that is called $\mathsf{E}^\bullet_{\mathscr{Z}/\mathbf{D}}$ in \cite{kkp2} is the graded sheaf $\Omega^\bullet_{\mathscr{Z}/\mathbf{D}}(\log \mathscr{D}_\infty,\mathrm{rel}\,f)$ equipped with the natural differential. A local computation in the proof of \cite[Claim 2.22]{kkp2} shows that the restriction of $\mathsf{E}^\bullet_{\mathscr{Z}/\mathbf{D}}$ to $Z \times \epsilon$ for $\epsilon \neq \infty$ is equal to $\Omega_Z^\bullet(\log D_\infty, \mathrm{rel}\, f^{-1}(\epsilon))$, and the restriction to $Z \times \infty$ gives the complex $\Omega_{Z}^\bullet(\log D_\infty, f)$. The complex $\mathsf{E}^\bullet_{\mathscr{Z}/\mathbf{D}}$ is a complex of analytic coherent sheaves on $\mathscr{Z}$.

The hyper-derived direct image $\mathbb{R}^ap_*\mathsf{E}_{\mathscr{Z}/\mathbf{D}}^\bullet$ has fibers which are just the hyper cohomology groups of the complexes $\Omega^\bullet_Z(\log D_\infty,\mathrm{rel}\, f^{-1}(\epsilon))$ if $\epsilon \neq \infty$ and $\Omega^\bullet_Z(\log D_\infty, f)$ if $\epsilon =\infty$. According to \cite[Lemma 2.21]{kkp2} or \cite[Appendix C]{esy} it is then true that the fibers of $\mathbb{R}^ap_*\mathsf{E}^\bullet_{\mathscr{Z}/\mathbf{D}}$ have constant dimension over $\mathbf{D}$ for all $a$. 

Now the $i^\mathrm{th}$ hypercohomology group of $\Omega_Z^\bullet(\log D_\infty, \mathrm{rel} \, f^{-1}(\epsilon))$ is isomorphic to the cohomology group $\HH^i(Y,f^{-1}(\epsilon);\mathbb{C})$, and the spectral sequence associated to the stupid filtration on it degenerates at the $E_1$ term. Thus we have that
$$
\dim H^{i}(Y,f^{-1}(\epsilon)) = \sum_{p+q=i}h^p(Z,\Omega^q(\log D_\infty,\mathrm{rel} \, f^{-1}(\epsilon))).
$$
Similarly by the degeneration of the Hodge-to-de Rham spectral sequence for $f$-adapted forms, (\cite[Lemma 2.19]{kkp2} or \cite[Theorem 1.3.2]{esy}), the same is true of $\mathbb{H}^i(Z,\Omega_Z^\bullet(\log D_\infty,f))$. In other words, 
$$
h^{i}(Y,w) = \sum_{p+q=i}h^p(Z,\Omega^q_Z(\log D_\infty, f)).
$$
The rest of our argument is standard. By Grauert's semicontinuity theorem, (see e.g. \cite[Theorem 8.5(ii)]{bpv}), the value of
$$
\epsilon \mapsto \dim \HH^p(p^{-1}(\epsilon), (\mathsf{E}^q_{\mathscr{Z}/\mathbf{D}})|_{p^{-1}(\epsilon)})
$$
is upper semicontinuous on $\mathbf{D}$ in the analytic Zariski topology. Thus it follows that for a general enough point $\epsilon_0$ of $\mathbf{D}$,
$$
h^{p,q}(Y,f^{-1}(\epsilon_0)) \leq f^{p,q}(Y,w).
$$
However, the fact that
\begin{align*}
\sum_{p+q=i}h^p(Z,\Omega^q(\log D_\infty,\mathrm{rel} \, f^{-1}(\epsilon_0))) &= \dim H^{i}(Y,f^{-1}(\epsilon_0)) \\ &= \dim \mathbb{H}^{i}(Z,\Omega^\bullet_Z(\log D_\infty, f)) \\ &= \sum_{p+q=i}\dim H^p(Z,\Omega^q(\log D_\infty, f)).
\end{align*}
implies that we must have equality between $h^{p,q}(Y,f^{-1}(\epsilon))$ and $f^{p,q}(Y,w)$ at all points.
\end{proof}
\begin{remark}\label{remk:additive}
The cohomology groups $H^i(Y,V;\mathbb{Q})$ admit a mixed Hodge structure whose Hodge filtration is given as above. Therefore
\begin{equation}\label{disambig}
h^{p,q}(Y,V) = \sum_{k} i^{p,k}(H^{p+q}(Y,V)).
\end{equation}
Combining Theorem \ref{thm:filt} with (\ref{disambig}) we find that 
$$
f^{p,q}(Y,w) = \sum_{k} i^{p,k}(H^{p+q}(Y,V)).
$$
This will be a very useful fact in Section \ref{ex:three}.
\end{remark}
\subsection{Poincar\'e duality}

We will now check that a version of Poincar\'e duality holds for $\HH^i(Y,w)$. Precisely, we will show that if $(Y,w)$ is a proper LG model of dimension $d$, then $h^{2d-i}(Y,w) = h^i(Y,w)$ for all $i$. First we recall the relative Mayer--Vietoris exact sequence.

\begin{proposition}
Let $Y_1$ and $Y_2$ be manifolds and let $S_1$ and $S_2$ be submanifolds of $Y_1$ and $Y_2$ respectively so that $Y = Y_1 \cup Y_2$ and let $S = S_1 \cup S_2 \subseteq Y$. Then there is a long exact sequence in cohomology,
$$
\dots \longrightarrow \HH^i(Y,S) \longrightarrow \HH^i(Y_1,S_1) \oplus \HH^i(Y_2,S_2) \longrightarrow \HH^i(Y_1 \cap Y_2, S_1 \cap S_2) \longrightarrow \dots.
$$
\end{proposition}
Let $\Sigma \subset \mathbb{C}$ be the set of critical values of $w$ and let $p$ be a base point in $\mathbb{C} \setminus \Sigma$. We may choose a collection $\{U_s\}_{s\in \Sigma}$ of open subsets of $\mathbb{C}$ which are homeomorphic to open discs so that 
\begin{enumerate}[--]
\item[--] each $U_s$ contains $p$ and $s \in \Sigma$ but no other critical values of $w$, 
\item[--] for any subset $S \subset \Sigma$, the set $\bigcap_{s\in S} U_s$ is simply connected, 
\item[--] $\bigcup_{s\in \Sigma} U_s$ is a deformation retract of $\mathbb{C}$. 
\end{enumerate}
Then let $Y_s = w^{-1}(U_s)$ for each $s \in \Sigma$. Let $V = w^{-1}(p)$. The following proposition is likely well known.
\begin{proposition}\label{prop:locglob}
Suppose that $Y,Y_s$ and $V$ are as above. Then
$$
h^i(Y,V) = \sum_{s\in \Sigma} h^i(Y_s,V).
$$
\end{proposition}
\begin{proof}
We will prove the case where $|\Sigma| = 2$. The general case is similar. Let $s_1,s_2 \in \Sigma$, then we have chosen $U_1$ and $U_2$ so that $U_1 \cap U_2$ is simply connected, open, and contains no critical points of $w$. Thus we have that $w^{-1}(U_1 \cap U_2)$ is a deformation retract onto $V$ by Ehresmann's theorem (see e.g. \cite[Theorem 9.3]{vois2}). This means that $\HH^i(w^{-1}(U_1 \cap U_2), V) = 0$. Therefore $\HH^i(Y_1\cup Y_2,V) \cong \HH^i(Y_1,V) \oplus \HH^i(Y_2,V)$ by the relative Mayer--Vietoris long exact sequence. 
\end{proof}
The following result is the main result of this subsection.
\begin{theorem}\label{thm:poincare}
If $(Y,w)$ is a proper LG model and $\dim Y = d$, then 
$$
h^{i}(Y,V) = h^{2d-i}(Y,V)
$$
for all $i$.
\end{theorem}
\begin{proof}
To each point, $s \in \Sigma$, we can associate a perverse sheaf of vanishing cycles $\phi_{w-s} \mathbb{C}$ supported on the critical locus of $w$ in $w^{-1}(s)$ (see e.g. \cite[Proposition 4.2.8]{dim}), and the hypercohomology of $\phi_{w-s}\mathbb{C}$ sits in a long exact sequence
$$
\dots \rightarrow \mathbb{H}^{i-1}(w^{-1}(s),\phi_{w-s}\mathbb{C}) \rightarrow \HH^i(Y_s;\mathbb{C}) \xrightarrow{r_i} \HH^i(V;\mathbb{C}) \rightarrow \mathbb{H}^{i}(w^{-1}(s),\phi_{w-s}\mathbb{C}) \rightarrow \dots
$$
where the map $r_i$ is the natural restriction map. However, this is precisely the map in the long exact sequence for relative cohomology, thus we find that
\[
\mathbb{H}^{i-1}(w^{-1}(s),\phi_{w-s} \mathbb{C}) \cong \HH^i(Y_s,V;\mathbb{C})
\]
and therefore,
\begin{equation}\label{eq:relcoh}
h^i(Y,V) = \sum_{s \in \Sigma} \dim\, \mathbb{H}^{i-1}(w^{-1}(s),\phi_{w-s}\mathbb{C}).
\end{equation}
Let $\mathbb{D}$ denote the Verdier duality functor. We know that $\mathbb{D}\mathbb{C}_{Y_s} = \mathbb{C}_{Y_s}[2d]$ (\cite[Example 3.3.8]{dim}) where $d$ is the complex dimension of $Y_s$. Furthermore, for any constructible complex $\mathscr{F}^\bullet$ on $Y_s$, $\mathbb{D} (\phi_{w-s}\mathscr{F}^\bullet [-1])\cong (\phi_{w-s} \mathbb{D}\mathscr{F}^\bullet)[-1]$ (\cite[Proposition 4.2.10]{dim}).  Using Verdier duality \cite[Theorem 3.3.10]{dim}, we see that
\begin{align*}
\mathbb{H}^{i+1}(w^{-1}(s) ,\phi_{w-s}\underline{\mathbb{C}}_{Y_s}) &\cong \mathbb{H}^{i}(w^{-1}(s) ,\phi_{w-s}\underline{\mathbb{C}}_{Y_s}[-1]) \\
&\cong \mathbb{H}_c^{-i}(w^{-1}(s), \mathbb{D}\phi_{w-s}\underline{\mathbb{C}}_{Y_s}[-1])^\vee\\ &\cong \mathbb{H}_c^{-i}(w^{-1}(s), \phi_{w-s}\underline{\mathbb{C}}_{Y_s}[2d-1])^\vee\\
&\cong  \mathbb{H}_c^{2d-(i+1)}(w^{-1}(s), \phi_{w-s}\underline{\mathbb{C}}_{Y_s})^\vee.
\end{align*}
Since $w^{-1}(s)$ is itself compact it follows that
\begin{equation}\label{eq:pdperv}
\mathbb{H}^{i}(w^{-1}(s), \phi_{w-s}\underline{\mathbb{C}}_{Y_s})\cong \mathbb{H}^{2d-i}(w^{-1}(s), \phi_{w-s}\underline{\mathbb{C}}_{Y_s})^\vee.
\end{equation}
Therefore, combining (\ref{eq:relcoh}), (\ref{eq:pdperv}), and Proposition \ref{prop:locglob}, the theorem follows.
\end{proof}
\begin{remark}
Theorem \ref{thm:poincare} requires surprisingly few assumptions about $(Y,w)$. We need $w$ to be a proper morphism and $Y$ to be smooth but nothing more.
\end{remark}
\subsection{Computing $h^{2}(Y,V)$ of a proper LG model}

%Every normal crossings variety has dual intersection complex $|\Gamma|$ which is a simplicial complex whose $i$-dimensional cells correspond to codimension $i$ strata in $X$ and adjacency in $|\Gamma|$ corresponds to inclusion of strata in $X$. 
%\begin{proposition}[\cite{mor}]
%Let $\Gamma$ be the dual intersection complex of a normal crossings variety $X$, and let $E^{i,j}_n$ denote the Mayer-Vietoris spectral sequence associated to $X$.  Then $E_2^{p,0}$ is isomorphic to $H^p(|\Gamma|;\mathbb{C})$. Therefore, $W_0H^p(X) \cong H^p(|\Gamma|;\mathbb{C})$.
%\end{proposition}
%
In this section, we will show that if $(Y,w)$ is a proper LG model and $V$ is a smooth fiber of $w$, then $h^2(Y,V) = h^2(Y,w)$ is an enumerative invariant of $(Y,w)$.
\begin{defn}\label{invk}
Let $(Y,w)$ be a proper LG model. Let $\Sigma$ be the set of critical values of $w$ and let $\rho_s$ be the number of irreducible components in $w^{-1}(s)$. Then
\[
k(Y,w) := \sum_{s\in \Sigma}(\rho_s-1).
\]
\end{defn}

The following result has been alluded to in work of Przyjalkowski and Shramov \cite{ps} and Przyjalkowski \cite{prz3}. In \cite{ps}, Przyjalkowski and Shramov begin with a smooth complete intersection Fano manifold $X$ of dimension $d$ in weighted projective space. They construct a proper LG model $(Y_X,w_X)$, which they expect to be mirror to $X$, and show that $k(Y_X,w_X) = h^{1,d-1}(X)$. In \cite{prz3}, Przyjalkowski proves a similar result for all Picard rank 1 Fano manifolds of dimension 3. In both cases, the following result is implicit.
\begin{theorem}\label{prop:folk}
Let $(Y,w)$ be a proper LG model of dimension $d$, let $V$ be a smooth fiber of $w$, and assume that $h^1(V)=0$. Then
$$
h^2(Y,V) = h^{2d-2}(Y,V) = k(Y,w).
$$
\end{theorem}
\begin{proof}
Recall that $\Sigma$ denotes the critical values of $w$. In Proposition \ref{prop:locglob}, we showed that if $H^i(Y,V)\cong \bigoplus_{s\in \Sigma}H^i(Y_s,V_s)$ where $Y_s$ is the preimage of a small disc $\mathbf{D}_s$ around $s \in \Sigma $ and $V_s$ is a generic smooth fiber above a point in $\mathbf{D}_s\setminus s$. Therefore, it's enough for us to show that $h^{2d-2}(Y_s,V_s) = \rho_s-1$. By Theorem \ref{thm:clcon}, $h^{2d-2}(Y_s) = h^{2d-2}(w^{-1}(s))$. 

Assume that $U_s :=w^{-1}(s)$ is normal crossings. We can now apply the Mayer--Vietoris spectral sequence (Theorem \ref{thm:mv}) to deduce that $H^{2d-2}(U_s)$ is a sum of subquotients of
$$
E^{i,2d-2-i}_1 = H^{2d-2-i}(U_s^{[i]}).
$$
However, $U_s^{[i]}$ has dimension $d-1-i$, so $H^{2d-2-i}(U_s^{[i]})=0$ if $i \neq 0$. Therefore $H^{2d-2}(Y_s)$ is a subquotient of $H^{2d-2}(U_s^{[0]})$. In particular, it is the kernel of
$$
H^{2d-2}(U_s^{[0]}) \longrightarrow H^{2d-2}(U_s^{[1]})
$$
which is just $H^{2d-2}(U_s^{[0]})$ for dimension reasons. Therefore, $h^{2d-2}(Y_s) = h^{2d-2}(U_s) = h^{2d-2}(U_s^{[0]}) = \rho_s$.

Now we can compute the dimension of the relative cohomology groups by the standard long exact sequence
$$
\dots \longrightarrow H^{2d-3}(V) \longrightarrow H^{2d-2}(Y_s,V) \longrightarrow H^{2d-2}(Y_s) \longrightarrow H^{2d-2}(V) \longrightarrow \dots.
$$
By monodromy invariance of $H^{2d-2}(V)$ and the local invariant cycle theorem, the map $H^{2d-2}(Y) \rightarrow H^{2d-2}(V)$ is surjective. By assumption, $H^{2d-3}(V) \cong H^1(V) \cong 0$. Therefore, it follows that $h^{2d-2}(Y_s,V) = \rho_s-1$.

Therefore if all fibers of $w$ are normal crossings then
$$
h^{2d-2}(Y,V) = \sum_{s\in \Sigma}h^{2d-2}(Y_s,V) = \sum_{s\in \Sigma}(\rho_s - 1).
$$
If the fibers of $w$ do not have normal crossings, then we can use Hironaka's theorem to blow up $Y$ repeatedly in connected smooth centers contained in fibers of $f$ to obtain a variety $\widetilde{Y}$ whose fibers have normal crossings. Let $\widetilde{w}$ be the composition of the morphism $\widetilde{Y} \rightarrow Y$ and $w$. Then $h^{2d-2}(Y) + k = h^{2d-2}(\widetilde{Y})$ where $k$ is the number of times we had to blow up $Y$. Furthermore, each blow up contributes one component to a singular fiber of $\widetilde{Y}$. Therefore, if $\widetilde{\rho}_s$ is the number of components of the fiber $\widetilde{w}^{-1}(s)$, then
\begin{equation}\label{numcompss}
\sum_{s\in \Sigma}(\widetilde{\rho}_s - 1) = k+ \sum_{s\in\Sigma}(\rho_s -1).
\end{equation}
The maps $H^{2d-2}(Y) \rightarrow H^{2d-2}(V)$ and $H^{2d-2}(\widetilde{Y}) \rightarrow H^{2d-2}(V)$ are both surjective and have kernel equal to $H^{2d-2}(Y,V)$ by the vanishing of $H^{2d-3}(V)$ so $h^{2d-2}(\widetilde{Y},V) = h^{2d+2}(\widetilde{Y}) -1$ and $h^{2d-2}(Y,V) = h^{2d-2}(Y) -1$. Thus 
\[
h^{2d-2}(Y,V)+ k = h^{2d-2}(\widetilde{Y},V) = \sum_{s\in \Sigma}(\widetilde{\rho}_s -1).
\]
Hence by (\ref{numcompss}) we have 
$$
h^{2d-2}(Y,V) = \sum_{s\in\Sigma}(\widetilde{\rho}_s - 1) - k = \sum_{s\in\Sigma}(\rho_s -1).
$$
as claimed.
\end{proof}

\subsection{The cohomology of $Y$}
Recall that in Section \ref{sect:relfilt}, we showed that the KKP Hodge numbers are the same as the dimensions of the Hodge-graded pieces of $H^i(Y,V;\mathbb{Q})$. The relative cohomology groups of $(Y,V)$ sit in a long exact sequence of mixed Hodge structures,
\[
\dots \longrightarrow H^i(Y,V;\mathbb{Q}) \longrightarrow H^i(Y;\mathbb{Q}) \longrightarrow H^i(V;\mathbb{Q}) \longrightarrow \dots.
\]
Therefore, to compute $f^{p,q}(Y,w)$, it will be enough to know $H^i(Y;\mathbb{Q}), H^i(V;\mathbb{Q})$, and the map $H^i(Y;\mathbb{Q}) \rightarrow H^i(V;\mathbb{Q})$ of mixed Hodge structures. In this section, we will show how to compute the mixed Hodge structure on $H^i(Y;\mathbb{Q})$ in terms of the cohomology of a tame compactification $(Z,f)$ of $(Y,w)$ and the limit mixed Hodge structure on $V$ at infinity.

More precisely, if $Y$ is of dimension $d$ then the mixed Hodge structure on $H^{2d-i}_c(Y;\mathbb{Q})$ is dual to that on $H^{i}(Y;\mathbb{Q})$ by work of Fujiki \cite[Theorem 1.7.1]{fujiki}, so we may compute $H^{2d-i}_c(Y;\mathbb{Q})$ instead of $H^{i}(Y;\mathbb{Q})$. Furthermore, if $(Z,f)$ is a tame compactification of a proper LG model $(Y,w)$ and $D_\infty = Z \setminus Y = f^{-1}(\infty)$, then there is a long exact sequence of mixed Hodge structures,
\[
\dots \longrightarrow H^i_c(Y;\mathbb{Q}) \longrightarrow H^i(Z;\mathbb{Q}) \rightarrow H^i(D_\infty;\mathbb{Q}) \longrightarrow \dots.
\]
We would like to understand this sequence.

Fix a proper LG model $(Y,w)$ and let $(Z,f)$ be a tame compactification. We will use $\mathbf{D}_\infty$ to denote a small disc in $\mathbb{P}^1$ containing $\infty$ and no other critical values of $f$. Let  $Y_\infty= f^{-1}(\mathbf{D}_\infty)$. By Theorem \ref{thm:clcon}, $Y_\infty$ admits a strong deformation retract to $D_\infty$, thus we may equip $H^i(Y_\infty;\mathbb{Q})$ with a mixed Hodge structure. If we choose $V$ to be a fiber over some point in $\mathbf{D}_\infty \setminus \infty$ and equip $H^i(V)$ with the limit mixed Hodge structure at $\infty$, then 
\[
H^i(Z;\mathbb{Q}) \xrightarrow{t_i} H^i(Y_\infty;\mathbb{Q}) \xrightarrow{r_i} H^i(V;\mathbb{Q})
\]
is a morphism of mixed Hodge structures obtained by pullback along the inclusions $V \subset Y_\infty \subset Z$. The map $r_i$ here is the same as the map $r_i$ in Theorem \ref{thm:csex}. For the remainder of this section, all cohomology groups will be taken with rational coefficients.
\begin{lemma}\label{lem:injti}
The kernel of $r_i$ is in the image of $t_i$.
\end{lemma}
\begin{proof}
If we let $\Sigma$ be the set of all singular values of $f$. Let $\mathbf{D}_s$ be a small disc around each $s\in \Sigma$ and let $Y_s = f^{-1}(\mathbf{D}_s)$, $Y_s^* = Y_s \setminus f^{-1}(s)$, and $Z_\Sigma = Z \setminus \sqcup_{s\in\Sigma}f^{-1}(s)$. We have a commutative diagram
\[
\begin{tikzcd}
H^i(Z,Z_\Sigma;\mathbb{Q}) \ar[r]\ar[d] & \bigoplus_{s\in \Sigma} H^i(Y_s, Y_s^*;\mathbb{Q}) \ar[d] \\
H^i(Z;\mathbb{Q}) \ar[r] & \bigoplus_{s \in \Sigma} H^i({Y_s};\mathbb{Q})
\end{tikzcd}
\]
The upper horizontal map is an isomorphism by excision. Therefore restricting to the fiber over $\infty$, we have another commutative diagram
\[
\begin{tikzcd}
H^i(Z,Z_\Sigma;\mathbb{Q}) \ar[r] \ar[d] & H^i(Y_\infty,Y_\infty^*;\mathbb{Q}) \ar[d] &  \\
H^i(Z;\mathbb{Q}) \ar[r] & H^i(Y_\infty;\mathbb{Q}) \ar[r] & H^i(V;\mathbb{Q})
\end{tikzcd}
\]
where the upper arrow is surjective and the diagram commutes. The image of the vertical right arrow is the kernel of lower right horizontal arrow by Theorem \ref{thm:csex}. The upper horizontal arrow is surjective and the image of the vertical right arrow is the kernel of $H^i(Y_\infty;\mathbb{Q}) \rightarrow H^i(V;\mathbb{Q})$. By commutativity of the square, this means that the image of
\[
H^i(Z,Z_\Sigma;\mathbb{Q}) \longrightarrow H^i(Z;\mathbb{Q}) \longrightarrow H^i(Y_\infty;\mathbb{Q})
\]
contains the kernel of $H^i(Y_\infty;\mathbb{Q}) \rightarrow H^i(V;\mathbb{Q})$. Therefore, the image of $H^i(Z;\mathbb{Q}) \rightarrow H^i(Y_\infty;\mathbb{Q})$ contains the kernel of $H^i(Y_\infty;\mathbb{Q}) \rightarrow H^i(V;\mathbb{Q})$.
\end{proof}
By  Theorem \ref{thm:csex}, $H^i(Y_\infty;\mathbb{Q}) \rightarrow H^i(V;\mathbb{Q})$ has image inside of $\ker N_i$, where $H^i(V;\mathbb{Q})$ is equipped with the limit mixed Hodge structure at $\infty$.
\begin{theorem}\label{thm:Yc}
Let $(Y,w)$ be a proper LG model. There is a short exact sequence of mixed Hodge structures
$$
0 \longrightarrow \mathscr{Q}_{i-1} \longrightarrow H^i_c(Y;\mathbb{Q}) \longrightarrow \mathscr{K}_i \longrightarrow 0
$$
where
$$
\mathscr{Q}_{i-1} = \mathrm{coker}(H^{i-1}(Z;\mathbb{Q}) \rightarrow H^{i-1}(Y_\infty;\mathbb{Q}) \rightarrow \ker N_{i-1})
$$
and $\mathscr{K}_i = \ker(H^i(Z;\mathbb{Q}) \rightarrow H^i(D_\infty;\mathbb{Q}))$.
\end{theorem}
\begin{proof}
Applying the long exact sequence for compactly supported cohomology for the triple $Z,D_\infty$ and $Y= Z \setminus D_\infty$, we see that there is a short exact sequence of mixed Hodge structures,
$$
0 \longrightarrow \mathscr{Q}_{i-1}' \longrightarrow H^i_c(Y;\mathbb{Q}) \longrightarrow \mathscr{K}_i \longrightarrow 0
$$
where $\mathscr{K}_i$ is the kernel of the restriction map $H^i(Z;\mathbb{Q}) \rightarrow H^i(D_\infty;\mathbb{Q})$ and $\mathscr{Q}_{i-1}'$ is the kernel of $H^{i-1}(Z;\mathbb{Q}) \rightarrow H^{i-1}(D_\infty;\mathbb{Q})$. Our goal is to show that $\mathscr{Q}_{i-1}'$ is isomorphic to $\mathscr{Q}_{i-1}$ as defined in the statement of the current proposition. For the sake of notation, we will talk about $\mathscr{Q}_i$ and $\mathscr{Q}_{i}'$ instead of $\mathscr{Q}_{i-1}$ and $\mathscr{Q}_{i-1}'$.

Denote by $C_i$ the image of the map $H^i(Z;\mathbb{Q}) \rightarrow H^i(Y_\infty;\mathbb{Q})$ and let $M_i$ be the kernel of $H^i(Y_\infty;\mathbb{Q}) \rightarrow H^i(V;\mathbb{Q})$. Then we have an injection of $M_i$ into $C_i$ by Lemma \ref{lem:injti} and the following commutative diagram with exact rows and columns
\begin{equation}\label{dag:3x3}
\begin{tikzcd}
& 0 \ar[d] &  0 \ar[d] & 0 \ar[d] &  \\
 0 \ar[r] & M_i \ar[r] \ar[d, "\mathrm{id}"] & C_i \ar[r]\ar[d] & C_i/M_i \ar[r] \ar[d] & 0 \\
 0 \ar[r] & M_i \ar[d] \ar[r] & H^i(Y_\infty;\mathbb{Q}) \ar[r]\ar[d] & \ker N_i \ar[r] \ar[d]& 0 \\
  &0 \ar[d] & \mathscr{Q}_i' \ar[d] & \mathscr{Q}_i \ar[d]& \\
  &0 & 0 & 0 & 
\end{tikzcd}
\end{equation}
where we have, by definition, $\mathscr{Q}_i' = \mathrm{coker}(t_i)$, and $\mathscr{Q}_i = \mathrm{coker}(C_i/M_i \rightarrow \ker N_i)$. Since the map from $C_i$ to $C_i/M_i$ is surjective, we can identify $\mathscr{Q}_i$ with the cokernel of the map $r_i \cdot t_i$. There is a homomorphism from $\mathscr{Q}_i'$ to $\mathscr{Q}_i$ which extends (\ref{dag:3x3}) to a commutative diagram
\begin{equation*}
\begin{tikzcd}
& 0 \ar[d] &  0 \ar[d] & 0 \ar[d] &  \\
 0 \ar[r] & M_i \ar[r] \ar[d, "\mathrm{id}"] & C_i \ar[r]\ar[d] & C_i/M_i \ar[r] \ar[d] & 0 \\
 0 \ar[r] & M_i \ar[d] \ar[r] & H^i(Y_\infty;\mathbb{Q}) \ar[r]\ar[d] & \ker N_i \ar[r] \ar[d]& 0 \\
0 \ar[r]  &0 \ar[d]\ar[r] & \mathscr{Q}_i' \ar[r] \ar[d] & \mathscr{Q}_i \ar[d]\ar[r]& 0 \\
  &0 & 0 & 0 & 
\end{tikzcd}
\end{equation*}
whose columns are exact and whose top two rows are exact. By the nine lemma, it then follows that the bottom row is exact, so there is an isomorphism of mixed Hodge structures from $\mathscr{Q}_i$ to $\mathscr{Q}_i'.$
\end{proof}

%-----------------------------------------------------------------------------------------------------------------------------

\section{The Hodge numbers of an LG model in three dimensions}\label{sect:hn3fold}\label{ex:three}

In this section, we will give a concrete method of computing the KKP Hodge numbers of certain LG models which appear as prospective mirrors of Fano threefolds.

%-----------------------------------------------------------------------------------------------------------------------------
% non-technical section

\subsection{Expected properties of LG mirrors of Fano threefolds}\label{sect:expectations}

In this subsection, Our goal is to define a certain class of LG models whose KKP Hodge numbers we will compute in subsequent subsections. The motivation for our definition comes from mirror symmetry, which we will now discuss briefly. The reader can find a more detailed, but still largely informal, explanation of the ideas in this subsection either in the work of Auroux \cite{aur1,aur2} or in the work of Katzarkov, Kontsevich, and Pantev \cite[Section 2]{kkp2}. For the purposes of our discussion, we will use the phrase ``log Calabi--Yau'' to refer to the complement of a reduced simple normal crossings anticanonical divisor in a smooth projective variety.

If one has a pair composed of a Fano manifold $X$ and a smooth anticanonical divisor $W$, then $X \setminus W$ is log Calabi--Yau.  We expect that this log Calabi--Yau manifold has a mirror log Calabi--Yau manifold $Y$ (see e.g. \cite{aur1}). Compactifying $X\setminus W$ to $X$ should be thought of as being equivalent to equipping $Y$ with a function $w$, so that $(Y,w)$ is mirror dual to $(X,s_W)$ in the sense of Definition \ref{defn:mirror}. Here $s_W$ is the global section of $\omega_X^{-1}$ whose vanishing locus is $W$. We expect that a smooth fiber $V$ of $w$ is mirror to $W$. The mirror to a smooth, compact, Calabi--Yau manifold is expected to be itself a smooth, compact, Calabi--Yau manifold, therefore, we can expect the fibers of $w$ to be compact.

Hence, we can expect that if $X$ is Fano and admits a smooth anticanonical divisor $W$, then $X$ admits a mirror LG model which is proper. All 3-dimensional Fano manifolds admit a smooth anticanonical section, so we expect that every Fano threefold admits a proper LG model as its mirror dual. If $X$ is a Fano threefold, then $W$ is a K3 surface. The mirror to a K3 surface is again a K3 surface, hence the fibers of $w$ should in general be projective K3 surfaces.

Let us assume that $X$ is a Fano manifold of dimension $d$, $(Y,w)$ is its mirror LG model, $W$ is a smooth anticanonical hypersurface in $X$, and $V$ is a smooth fiber of $w$. Let $T_{d-1} \in \mathrm{GL}(H^{d-1}(V;\mathbb{C}))$ be the monodromy automorphism associated to a small counterclockwise loop around $\infty$. Then, by (\ref{hnms}) one expects that $H^{d-1}(V;\mathbb{C}) \cong \bigoplus_{i=0}^{d-1} H^{i,i}(W)$. A general expectation from mirror symmetry is that the action of $N_{d-1}=\log \,T_{d-1}\in \mathrm{End}(H^{d-1}(V;\mathbb{C}))$ should be identified with the cup product $(-)\cup c_1(X)|_W$ \cite{kkp2}. Since $X$ is Fano (thus $c_1(X)$ is an ample class), the hard Lefschetz theorem implies that the operator $(-)\cup c_1(X)|_W$ has the property that $((-)\cup c_1(X)|_W)^{d} =0$ but $((-1)\cup c_1(X)|_W)^{d-1} \neq 0$. Therefore, if $(Y,w)$ is mirror to $X$, we expect that $N_{d-1}^{d} =0$ but $N_{d-1}^{d-1} \neq 0$. If $N_{d-1}$ satisfies these conditions we say that $T_{d-1}$ is {\em maximally unipotent}. 

Summarizing this discussion, if $X$ is a Fano threefold, we expect that there is an LG model $(Y,w)$ which is mirror to $X$ and so that:
\begin{enumerate}[--]
\item[--] $w$ is proper,
\item[--] the smooth fibers of $w$ are K3 surfaces,
\item[--] the monodromy action on $H^2(V)$ around infinity is maximally unipotent,
\item[--] $(Y,w)$ admits a tame compactification $(Z,f)$ so that $f^{-1}(\infty)$ is an anticanonical hypersurface.
\end{enumerate}
The monodromy of families of K3 surfaces is well understood \cite[\S 4(d)]{mor} and as a result, we may impose a geometric condition on the tame compactification $(Z,f)$ which forces monodromy around infinity to be maximally unipotent.
\begin{defn}[Kulikov, \cite{kul}]
Let $g: \mathscr{U} \rightarrow \mathbf{D}$ be a projective fibration over a complex disc containing 0 whose fibers away from 0 are smooth K3 surfaces. We say that $g$ is a {\em semistable type III degeneration of K3 surfaces} if 
\begin{enumerate}
\item$K_{\mathscr{U}}$ is trivial,
\item$g^{-1}(0)$ is a simple normal crossings union of smooth rational surfaces whose dual intersection complex is a triangulation of $S^2$.
\end{enumerate}
\end{defn}
Suppose that $g:\mathscr{U}\rightarrow \mathbf{D}$ is a type III degeneration of K3 surfaces. Then $N_2$ has one nontrivial Jordan block of rank 2 and its remaining Jordan blocks are trivial (see e.g. \cite{kul,fs}). In particular, $T_2$ is maximally unipotent.
\begin{remark}\quad
\begin{enumerate}
\item There exist degenerations $g :\mathscr{U} \rightarrow \mathbf{D}$ of K3 surfaces with maximally unipotent monodromy which cannot be modified birationally to produce a type III degeneration of K3 surfaces. However, examples of this seem to be rare in practice. 
\item As the notation {\em type III} indicates, there are semistable type I and II degenerations of K3 surfaces as well (see \cite{kul} for details). In the semistable type I case, monodromy is trivial. In the semistable type II case, $N_2$ has the property that $N_2^2 = 0$ but $N_2 \neq 0$. Therefore, these types of degenerations should not appear in LG models which are mirror to Fano threefolds.
\end{enumerate}
\end{remark}

Let $(Y,w)$ be a proper LG model and assume that $\dim Y = 3$. Let $(Z,f)$ be a tame compactification of $(Y,w)$.  In the examples that we will look at, $Z$ is rational hence we have that
\[
h^{1,0}(Z) = h^{2,0}(Z) = h^{3,0}(Z) = 0.
\]
Therefore if $X$ is a Fano threefold then $X$ will usually admit a mirror LG model $(Y,w)$ which has a tame compactification of the following type.
\begin{defn}\label{defn:wf}
Let $(Y,w)$ be a proper LG model and assume that $\dim Y = 3$. We say that a tame compactification $(Z,f)$ of $(Y,w)$  is {\em type III} if 
\begin{enumerate}
\item $h^{i,0}(Z) = 0$ for $i \neq 0$,
\item the fiber $D_\infty$ over $\infty$ of $f$ is a type III degeneration of K3 surfaces
\end{enumerate}
\end{defn}
\begin{remark}
The invariants $h^{i,0}(Z)$ are birational invariants of $Z$, therefore if (1) holds for any tame compactification, it holds for any tame compactification. Condition (2) is essentially a relative minimality condition. If $(Z,f)$ is a type III compactification, one can sometimes produce other type III compactifications by making specific birational modifications, so type III compactifications are not unique in general.
\end{remark}
\begin{remark}
In \cite{prz}, Przyjalkowski shows that for all Fano threefolds with very ample anticanonical bundle there are (prospective) mirror LG models which admit type III compactifications. In Section \ref{sect:toric}, we will show that all weak Fano toric threefolds have (prospective) LG mirrors which admit type III compactifications.
\end{remark}

Before proceeding to the main theorem in this section, we will explain a general fact about type III degenerations of K3 surfaces. We remark that the most difficult part of this result follows from work of Friedman and Scattone \cite[Proposition 7.2]{fs}.
\begin{lemma}\label{lemma:t3}
Let $g:\mathscr{U} \rightarrow \mathbf{D}$ be a type III degeneration of K3 surfaces and let $U$ be $g^{-1}(0)$ be the degenerate fiber of $g$. Then $H^1(U) \cong  H^3(U) \cong 0$.
\end{lemma}
\begin{proof}
By the Mayer--Vietoris spectral sequence, $H^1(U)$ is isomorphic to $E_2^{1,0} \oplus E_2^{0,1} $. The group $E^{0,1}_2$ is a subquotient of $E_1^{0,1} = H^1(U^{[0]})$ which vanishes because all components of $U$ are rational surfaces and hence $H^1(U_i) = 0$ for all $i$. According to Proposition \ref{cor:mor}, $E^{p,0}_2 = H^p(|\Gamma|)$ where $|\Gamma|$ is the dual intersection complex of $U$. In our case $|\Gamma|$ is homeomorphic to $S^2$ so it follows that $E^{1,0}_2 = 0$. Therefore, $H^1(U)$ must be $0$. 

For the same reasons, $H^3(U)$  is isomorphic to a direct sum of subquotients of the groups $E^{0,3}_1 = H^3(U^{[0]}), E_{1}^{1,2} = H^2(U^{[1]})$ and $E_{1}^{2,1} = H^1(U^{[2]})$. The first vanishes since the components of $U$ are rational and the third vanishes since $U^{[2]}$ is a union of points. It is shown by Friedman and Scattone \cite[Proposition 7.2]{fs} that the differential $d:E_1^{0,2} \rightarrow E_1^{1,2}$ is surjective. Therefore $E_2^{1,2}=0$ and hence $H^3(U)=0$. 
\end{proof}

\subsection{The main theorem}
In this subsection, we will compute the KKP Hodge numbers of LG models which admit type III compactifications. We introduce the following notation.
\begin{defn}\label{def:ph}
Let $(Y,w)$ be a proper LG model so that $Y$ is of dimension 3, and let $V$ be a smooth fiber of $w$. Then $ph(Y,w)$ will denote the dimension of the cokernel of
\[
H^2(Y) \longrightarrow H^2(V).
\]
\end{defn}
The rest of this section will be devoted to proving the following result.
\begin{theorem}\label{thm:3fold}
Let $(Y,w)$ be a proper LG model which admits a type III compactification $(Z,f)$. Then
\begin{align*}
f^{3,0}(Y,w) & = f^{0,3}(Y,w) = 1,\\ 
f^{1,1}(Y,w) &= f^{2,2}(Y,2) = k(Y,w),\\  
f^{2,1}(Y,w) &= f^{1,2}(Y,w) = ph(Y,w) - 2 + h^{2,1}(Z)
\end{align*}
where $k(Y,w)$ is as in Definition \ref{invk}.
\end{theorem}
The proof of Theorem \ref{thm:3fold} will be given over the course of  several lemmas and propositions in Subsection \ref{sect:proof3fold}. Let us briefly discuss Theorem \ref{thm:3fold} before proceeding with its proof. If $(Y,w)$ is a proper LG model in dimension 3 which admits a type III compactication, we may arrange the invariants $f^{p,q}(Y,w)$ into a ``Hodge diamond''. Theorem \ref{thm:3fold} says that this Hodge diamond takes the following form.
$$
\begin{matrix}
&&&0&&& \\
&&0&&0&&\\
&0&& k(Y,w) &&0& \\
1\qquad \, \,&& ph(Y,w)-2 + h^{1,2}(Z) && ph(Y,w)-2 + h^{2,1}(Z) &&\qquad \,\,1 \\
&0&& k(Y,w) &&0& \\
&&0&&0&&\\
&&&0&&&
\end{matrix}
$$
Therefore, the Hodge diamond of an LG model admitting a type III compactification looks like the Hodge diamond of a Fano threefold reflected across the appropriate diagonal. This should be seen as support for the fact that the invariants $f^{p,q}(Y,w)$ are appropriate for studying mirror symmetry.

\subsection{The proof of Theorem \ref{thm:3fold}}\label{sect:proof3fold}
Roughly, the proof of Theorem \ref{thm:3fold} consists of two separate parts. First, we will show that $f^{p,q}(Y,w) = 0$ if $p + q \neq 3$ and $(p,q) \neq (1,1)$ or $(2,2)$. Finally we will compute $f^{p,3-p}(Y,w)$.

\begin{proposition}\label{prop:vanish1}
Let $(Y,w)$ be a proper LG model, let $\dim Y = 3$ and assume that $(Y,w)$ admits a type III compactification. Then 
\[
H^0(Y,V) \cong H^1(Y,V) \cong H^5(Y,V) \cong H^6(Y,V) \cong 0.
\]
Hence, $f^{p,q}(Y,w) = 0$ if $p+q \leq 1$ or $p+q \geq 5$.
\end{proposition}
\begin{proof}
By Theorem \ref{thm:poincare}, this is equivalent to showing that $H^5(Y,V)$ and $H^6(Y,V)$ vanish. Recall the long exact sequence in relative cohomology
\begin{equation}\label{eq:les}
\dots \longrightarrow H^{i}(V) \longrightarrow H^{i+1}(Y,V) \longrightarrow H^{i+1}(Y) \longrightarrow \dots.
\end{equation}
By Theorem \ref{thm:Yc}, we have that $H^0_c(Y)$ is isomorphic to the kernel of
\begin{equation}\label{eq:identify}
H^0(Z) \longrightarrow H^0(D_\infty).
\end{equation}
This map is an isomorphism, so $H^0_c(Y) \cong H^6(Y) = 0$. Therefore, by (\ref{eq:les}) and the fact that $H^5(V) = 0$, we can conclude that $H^6(Y,V) = 0$. Using Theorem \ref{thm:poincare}, we may conclude that that $H^0(Y,V)=0$ as well.

We have already noted that (\ref{eq:identify}) is an isomorphism, therefore $\mathscr{Q}_0=0$, so $H^1_c(Y) \cong \mathscr{K}_1$. By assumption, $H^1(Z)=0$, hence $\mathscr{K}_1 =0$. Therefore $H^1_c(Y) \cong H^5(Y)=0$. Appealing to Equation (\ref{eq:les}), we then have that $H^5(Y,V)$ is isomorphic to the cokernel of the restriction map from $H^4(Y)$ to $H^4(V)$. Since $H^4(V)$ is monodromy invariant, it follows from the local invariant cycles theorem that this cokernel is trivial and hence $H^5(Y,V)=0$ and therefore, $H^2(Y,V) = 0$ by Proposition \ref{thm:poincare}.
\end{proof}
\begin{defn}
A mixed Hodge structure $(R,F,W)$ is said to be {\em Hodge--Tate} if $i^{p,q}(R) = 0$ if $p \neq q$.
\end{defn}
\begin{proposition}\label{prop:lowhn}
Let $(Y,w)$ be a proper LG model, let $\dim Y = 3$ and assume that $(Y,w)$ admits a type III compactification. Then 
the mixed Hodge structures on $H^2(Y,V;\mathbb{Q})$ and $H^4(Y,V;\mathbb{Q})$ are pure Hodge structures of weight 2 and 4 respectively which are Hodge--Tate. In other words, $h^{p,2-p}(Y,V) = 0$ if $p \neq 1$, and $h^{p,4-p}(Y,V) = 0$ if $p \neq 2$.
\end{proposition}
\begin{proof}
By Lemma \ref{lemma:t3} and the fact that $D_\infty$ is a type III degeneration of K3 surfaces, $H^1(D_\infty;\mathbb{Q}) \cong 0$. Therefore, $\mathscr{Q}_1 \cong 0$. By the assumption that $Z$ is smooth and projective, $h^{2,0}(Z) = h^{0,2}(Z)$, so it follows that $H^2_c(Y;\mathbb{Q})$ (which is isomorphic to $\mathscr{K}_2$) has $h^{p,q}(H^2_c(Y;\mathbb{Q})) \neq 0$ only if $(p,q) = (1,1)$. Duality between $H_c^2(Y;\mathbb{Q})$ and $H^4(Y;\mathbb{Q})$ then tells us that the only possibly nonzero Deligne--Hodge number of $H^4(Y;\mathbb{Q})$ is $i^{2,2}$. Since $H^3(V;\mathbb{Q}) = 0$ it follows that $H^4(Y,V;\mathbb{Q})$ is a quotient of $H^4(Y;\mathbb{Q})$. Therefore, $H^4(Y,V;\mathbb{Q})$ is, pure, Hodge--Tate, and of weight 4.

A similar argument applies for $H^2(Y,V;\mathbb{Q})$. By Lemma \ref{lemma:t3}, $H^3(D_\infty;\mathbb{Q}) \cong 0$. Therefore, $\mathscr{Q}_3 = 0$ and $H^4_c(Y;\mathbb{Q}) \cong \mathscr{K}_4$. Since $h^{p,q}(H^4(Z;\mathbb{Q})) = 0$ unless $(p,q) = (2,2)$, the same is then true for $H^4_c(Y;\mathbb{Q})$. Duality between $H_c^4(Y;\mathbb{Q})$ and $H^2(Y;\mathbb{Q})$ then tells us that $H^2(Y;\mathbb{Q})$ carries a pure Hodge structure with the property that $i^{p,q}=0$ unless $(p,q) = (1,1)$. Since $H^1(V;\mathbb{Q}) = 0$, it follows that $H^2(Y,V;\mathbb{Q})$ also has the property that $i^{p,q}=0$ unless $(p,q) = (1,1)$. 
\end{proof}
\begin{corollary}\label{cor:vanish2}
Let $(Y,w)$ be a proper LG model, let $\dim Y = 3$ and assume that $(Y,w)$ admits a type III compactification. Then 
\[
f^{2,0}(Y,w) = f^{0,2}(Y,w) = f^{1,3}(Y,w) = f^{3,1}(Y,w) = 0.
\]
\end{corollary}
\begin{proof}
Apply Proposition \ref{prop:lowhn} and Theorem \ref{thm:filt}.
\end{proof}
\begin{corollary}\label{cor:vanish3}
Let $(Y,w)$ be a proper LG model, let $\dim Y = 3$ and assume that $(Y,w)$ admits a type III compactification. Then $f^{1,1}(Y,w) = f^{2,2}(Y,w) = k(Y,w)$.
\end{corollary}
\begin{proof}
By \cite[Lemma 2.19, Lemma 2.21]{kkp2}, we have that $f^{2,0}(Y,w) + f^{1,1}(Y,w)  + f^{0,2}(Y,w) = h^{2}(Y,V)$, therefore, Proposition \ref{prop:lowhn}, along with Theorem \ref{thm:filt} implies that $f^{1,1}(Y,w) = h^2(Y,V)$. By Theorem \ref{prop:folk}, $h^{2}(Y,V) = k(Y,w)$, hence $f^{1,1}(Y,w) = k(Y,w)$. A nearly identical argument shows that $f^{2,2}(Y,w) = k(Y,w)$.
\end{proof}

\begin{lemma}\label{lem:ses2}
Let $(Y,w)$ be a proper LG model, let $\dim Y = 3$, and assume that $(Y,w)$ admits a type III compactification $(Z,f)$. Then
\begin{align*}
i^{3,3}(H^3(Y;\mathbb{Q})) &= 1, \quad i^{2,2}(H^3(Y;\mathbb{Q})) = ph(Y,w) - 3, \\ i^{1,2}(H^3(Y;\mathbb{Q})) &= i^{2,1}(H^3(Y;\mathbb{Q}))= h^{1,2}(Z).
\end{align*}
\end{lemma}
\begin{proof}
By Theorem \ref{thm:Yc}, we have a short exact sequence of mixed Hodge structures,
\begin{equation}\label{eq:blah}
0 \longrightarrow \mathscr{Q}_2\longrightarrow H_c^3(Y) \longrightarrow \mathscr{K}_3 \longrightarrow 0.
\end{equation}
According to Lemma \ref{lemma:t3}, $H^3(D_\infty;\mathbb{Q}) = 0$. Therefore $\mathscr{K}_3 \cong H^3(Z,\mathbb{Q})$. Thus to compute $i^{p,q}(H^3_c(Y;\mathbb{Q}))$, we just need to compute $\mathscr{Q}_2$ along with its mixed Hodge structure. It is well known (see e.g. \cite[pp. 113]{mor}) that the limit mixed Hodge structure of a type III degeneration of K3 surfaces satisfies $W_0 =W_1 \cong \mathbb{Q}, W_2 = W_3 \cong \mathbb{Q}^{21}$, $W_4 \cong \mathbb{Q}^{22}$ and $F^2 \cong \mathbb{C}, F^1 \cong \mathbb{C}^{21}$ and $F^0 \cong \mathbb{C}^{22}$. Let $H_\mathrm{lim}$ denote $H^2(V;\mathbb{Q})$ equipped with the limit mixed Hodge structure at infinity. Then
$$
i^{0,0}(H_\mathrm{lim}) = 1, \quad i^{1,1}(H_\mathrm{lim}) =20, \quad i^{2,2}(H_\mathrm{lim}) = 1
$$
and $i^{p,q}(H_\mathrm{lim}) = 0$ otherwise. If $T_2$ is the monodromy transformation acting on $H^2(V;\mathbb{Q})$ coming from a counterclockwise rotation around $\infty$ and $N_2 = \log T_2$, then $\ker N_2$ maps $H_\mathrm{lim}$ to $H_\mathrm{lim}(-1)$. Thus we obtain maps
$$
n_i :\mathrm{Gr}_F^i\mathrm{Gr}^W_{2i}H_\mathrm{lim}  \longrightarrow \mathrm{Gr}_F^{i-1} \mathrm{Gr}^W_{2i-2}H_\mathrm{lim}(-1)
$$
for $i=0,1,2$ induced by $N_2$. The map $n_2$ is injective, $n_1$ is surjective and $n_0$ is zero. Hence $\ker n_2 \cong 0, \ker n_1 \cong \mathbb{C}^{19}$ and $\ker n_2 \cong \mathbb{C}$. Therefore
$$
i^{2,2}(\ker N_2) = 0, \quad i^{1,1}(\ker N_2) = 19, \quad i^{0,0}(\ker N_2) = 1
$$
and $h^{p,q}(\ker N_2) = 0$ otherwise. By the global invariant cycles theorem the image of the restriction map $H^2(Z;\mathbb{Q}) \rightarrow H^2(V;\mathbb{Q})$ has image which is invariant under $N_2$. Since the action of $N_2$ on $H_\mathrm{lim}$ is induced by the action of $N_2$ on $H^2(V;\mathbb{Q})$, it follows that the image of the map $H^2(Z;\mathbb{Q})\rightarrow H^2(V;\mathbb{Q})$ is contained in $\ker N_2$. Therefore, $\mathscr{Q}_2$ satisfies
$$
i^{1,1}(\mathscr{Q}_2) = ph(Y,w) - 3, \quad i^{0,0}(\mathscr{Q}_2) = 1
$$
and $h^{p,q}(\mathscr{Q}_2) = 0$ otherwise. Therefore, by additivity of Hodge numbers in exact sequences, the nonzero Hodge numbers of $H^3_c(Y;\mathbb{Q})$ are
\begin{align*}
i^{0,0}(H^3_c(Y;\mathbb{Q})) &= 1, \quad i^{1,1}(H^3_c(Y;\mathbb{Q})) = ph(Y,w) - 3,\\  i^{2,1}(H^3_c(Y;\mathbb{Q})) &= i^{1,2}(H^3_c(Y;\mathbb{Q}))=h^{1,2}(Z).
\end{align*}
Since we have that $H_c^3(Y;\mathbb{Q})^\vee(-3)$ is isomorphic as a mixed Hodge structure to $H^3(Y;\mathbb{Q})$, we see that $i^{p,q}(H^3_c(Y;\mathbb{Q})) = i^{3-p,3-q}(H^3(Y;\mathbb{Q}))$. Thus the lemma follows as claimed.
\end{proof}

\begin{proposition}\label{prop:hodgy}
If $(Y,w)$ is a proper LG model which admits a type III compactification $(Z,f)$ then
\begin{align*}
f^{3,0}(Y,w) = f^{0,3}(Y,w) = 1, \quad f^{2,1}(Y,w) = f^{1,2}(Y,w) = ph(Y,w)-2 + h^{1,2}(Z).
\end{align*}
\end{proposition}
\begin{proof}
According to Lemma \ref{lem:ses2}, the nonzero Hodge numbers of $H^3(Y)$ are
\begin{align*}
i^{3,3}(H^3(Y;\mathbb{Q})) &= 1, \quad i^{2,2}(H^3(Y;\mathbb{Q}))  = ph(Y,w) - 3, \\ i^{2,1}(H^3(Y;\mathbb{Q})) &= i^{1,2}(H^3(Y;\mathbb{Q}))  = h^{1,2}(Z).
\end{align*}
Now we may compute the Hodge--Deligne numbers of $H^3(Y,V)$. We have the long exact sequence
$$
\dots \longrightarrow H^2(Y;\mathbb{Q}) \longrightarrow H^2(V;\mathbb{Q}) \longrightarrow H^3(Y,V;\mathbb{Q}) \longrightarrow H^3(Y;\mathbb{Q}) \longrightarrow 0.
$$
By Proposition \ref{prop:lowhn}, $H^2(Y;\mathbb{Q})$ is purely of type $(1,1)$, so the nonzero Hodge numbers of the cokernel of $H^2(Y;\mathbb{Q})\rightarrow H^2(V;\mathbb{Q})$ (which we will call $PH$) are
$$
i^{2,0}(PH) = i^{0,2}(PH) = 1, \qquad i^{1,1}(PH) = ph(Y,w)-2.
$$
This uses the fact that the image of the restriction map of $H^2(Y;\mathbb{Q})$ to $H^2(V;\mathbb{Q})$ is the same as that of the restriction map from $H^2(Z;\mathbb{Q})$ to $H^2(V;\mathbb{Q})$ by the global invariant cycles theorem (Theorem \ref{thm:gict}). Since Hodge numbers are additive in short exact sequences of mixed Hodge structures, it follows that the nonzero Hodge numbers of $H^3(Y,V)$ are
\begin{align*}
i^{3,3}(H^3(Y,V&;\mathbb{Q})) = 1, \, i^{2,2}(H^3(Y,V;\mathbb{Q})) = ph(Y,w) - 3, \, i^{1,1}(H^3(Y,V;\mathbb{Q})) = ph-2, \\
 i^{2,1}(H^3(Y,V;\mathbb{Q})) = &\,i^{1,2}(H^3(Y,V;\mathbb{Q})) = h^{1,2}(Z), \, i^{2,0}(H^3(Y,V;\mathbb{Q})) = i^{0,2}(H^3(Y,V)) = 1.
\end{align*}
Therefore, the fact that $\dim \mathrm{Gr}_F^pH^3(Y,V;\mathbb{C}) = \sum_{q} i^{p,q}(H^3(Y,V;\mathbb{Q}))$ (as noted in Remark \ref{remk:additive}) implies that
\[
\dim \mathrm{Gr}_F^0H^3(Y,V;\mathbb{C}) = \dim\Gr_F^3H^3(Y,V;\mathbb{C})= 1
\]
and
\[
\dim\Gr_F^1H^3(Y,V;\mathbb{C}) =\dim \Gr_F^2H^3(Y,V;\mathbb{C}) = ph(Y,w) - 2 + h^{1,2}(Z).
\]
This finishes the proof of the proposition when combined with Theorem \ref{thm:filt}.
\end{proof}
Theorem \ref{thm:3fold} is then a combination of Proposition \ref{prop:vanish1}, Corollary \ref{cor:vanish2}, Corollary \ref{cor:vanish3}, and Proposition \ref{prop:hodgy}.

\section{Mirror symmetry for weak toric Fano threefolds}\label{sect:toric}

In the next section, we will establish Conjecture \ref{conj:KKP} for crepant resolutions of Gorenstein toric Fano varieties and their LG mirrors.

\subsection{Toric weak Fano threefolds}\label{sect:constor}
Now let us discuss mirror symmetry for smooth toric threefolds. Our notation is similar to that of Batyrev \cite{bat} and Batyrev-Borisov \cite{bat-bo}, except for the fact that we have inverted $\Delta$ and $\Delta^\circ$. The reader is referred to \cite{bat} and the book of Cox, Little, and Schenck \cite{cls} for more details. We will assume that the reader is familiar with basic notions of fans and polytopes.

Let us review some basic facts regarding toric varieties. Following standard notation, $M$ will always denote a lattice (a free abelian group of finite rank), and $N = \mathrm{Hom}(M,\mathbb{Z})$ will denote its dual. For any fan $\Sigma$ in $M\otimes \mathbb{R}$, we will let $\Sigma[n]$ denote the disjoint union of all cones of dimension $n$. The rays of $\Sigma$ are elements of the set $\Sigma[1]$.

\begin{defn}
Let $\Delta \subseteq M \otimes \mathbb{R}$ be a bounded polytope. We say $\Delta$ is {\em integral} if all of its vertices are at points in $M \subseteq M \otimes \mathbb{R}$. Then
\[
\Delta^\circ = \{ \rho \in \mathrm{Hom}(M,\mathbb{R})\cong N \otimes \mathbb{R} : \rho(x) \leq -1 \, \, \forall \, \, x \in \Delta \}.
\] 
is called the {\em polar dual} of $\Delta$. We say $\Delta$ is {\em reflexive} if both $\Delta$ and $\Delta^\circ$ are integral and the origin is in the interior of $\Delta$. If $F$ is a face of $\Delta$ of dimension $i$, then there is a dual face $F^\circ$ of $\Delta^\circ$ of dimension $\rank M - 1 -i$ defined to be 
\[
\{\rho \in \Delta^\circ, \rho(x) = -1 : x \in F \}.
\]
\end{defn}

Henceforward, $\Delta$ will denote a reflexive polytope embedded in $M \otimes \mathbb{R}$ for some lattice $M$ of rank 3. Let $\Sigma(\Delta)$ be the fan over the faces of $\Delta$. The cones of $\Sigma(\Delta)$ are in bijection with the faces of $\Delta$. 

We may choose a refinement of $\Sigma(\Delta)$, which we call $\hat{\Sigma}(\Delta)$, so that each 3-dimensional cone of $\hat{\Sigma}$ is spanned by rays which generate the lattice $M$, and that the rays of $\hat{\Sigma}(\Delta)$ are generated by the integral points in the boundary of $\Delta$, which we will denote $\partial \Delta$. By abuse of notation, we may identify rays $\rho$ with their generators in $\partial \Delta \cap M$. Such a refinement exists, according to Batyrev \cite{bat}, but is not necessarily unique. Such a refinement induces a triangulation of the boundary $\partial \Delta$ of $\Delta$ whose vertex sets contains all integral points in $\partial \Delta$. The results of this section are independent of our choice of $\hat{\Sigma}(\Delta)$.

Let $\hat{X}_\Delta$ be the toric variety associated to the fan $\hat{\Sigma}(\Delta)$, then $\hat{X}_\Delta$ is smooth, projective, and weak Fano.  To each ray $\rho$ of $\hat{\Sigma}(\Delta)[1]$, there is a divisor $D_\rho$ of $\hat{X}_\Delta$. A standard result in toric geometry (see, e.g. \cite[Proposition 4.4.1]{bat}) says that the divisors $D_\rho$ generate $\mathrm{Pic}(\hat{X}_\Delta)$ and have $\rank M$ relations. More precisely, there is a short exact sequence of groups,
\begin{equation}\label{eq:picses}
0 \longrightarrow \mathrm{Hom}_\mathbb{Z}(M,\mathbb{Z}) \longrightarrow \mathbb{Z}^{\partial \Delta\cap M} \longrightarrow \mathrm{Pic}(\hat{X}_\Delta) \longrightarrow 0
\end{equation}
where $n \in N = \Hom_\mathbb{Z}(M,\mathbb{Z})$ is mapped to $\sum_{\rho \in \partial\Delta\cap M} n(\rho)[\rho]$. Therefore,
$$
\rank \Pic(\hat{X}_\Delta) = h^{1,1}(\hat{X}_\Delta) = \ell(\Delta) - 4.
$$
where $\ell(\Delta)$ is the number of lattice points in $\Delta$. Note that $\ell(\Delta)-1$ the number of 1-dimensional strata of $\hat{\Sigma}(\Delta)$, since for any reflexive polytope, the only interior point is the origin.

Any toric variety can be algebraically stratified into a number of copies of $(\mathbb{C}^\times)^k$ so $h^{p,q}(\hat{X}_\Delta) = 0$ if $p \neq q$. Therefore $h^{2,1}(\hat{X}_\Delta) = 0$ and we can compute the Hodge numbers of $\hat{X}_\Delta$.
\begin{proposition}\label{prop:Xdhn}
The Hodge numbers of $\hat{X}_\Delta$ are 
\[
h^{p,q}(\hat{X}_\Delta) = \begin{cases}
1 & \mathrm{if } \quad (p,q) = (0,0), (3,3) \\
\ell(\Delta) - 4 & \mathrm{if }\quad (p,q) = (1,1), (2,2) \\
0 & \mathrm{otherwise.}
\end{cases}
\]
\end{proposition}
Applying the functor $\mathrm{Hom}_\mathbb{Z}(-,\mathbb{C}^\times)$ to (\ref{eq:picses}), we get a short exact sequence of tori,
\[
0 \longrightarrow G_{\hat{\Sigma}(\Delta)} =\mathrm{Hom}_\mathbb{Z}(\mathrm{Pic}(\hat{X}_\Delta),\mathbb{C}^\times) \longrightarrow (\mathbb{C}^\times)^{\partial \Delta \cap M} \longrightarrow T_M=M \otimes \mathbb{C}^\times \longrightarrow 0.
\]
One may partially compactify $(\mathbb{C}^\times)^{\partial \Delta \cap M}$ to a quasi affine variety $V_{\hat{\Sigma}(\Delta)} \subseteq \mathbb{C}^{\partial \Delta \cap M}$ (which depends on $\hat{\Sigma}(\Delta)$) so that $\hat{X}_{\Delta} = V_{\hat{\Sigma}(\Delta) }/\mkern-6mu/ G_{\hat{\Sigma}(\Delta)}$. This equips $\hat{X}_\Delta$ with a homogeneous coordinate ring $\mathbb{C}[\{x_\rho\}_{\rho \in \partial \Delta \cap M}]$ which is graded by $\mathrm{Pic}(\hat{X}_\Delta)$ \cite[Theorem 5.1.11]{cls}. The global sections of the anticanonical bundle, $\omega^{-1}_{\hat{X}_\Delta}$, may be identified with homogeneous polynomials
\[
\sum_{\alpha \in \Delta^\circ \cap N} a_\alpha \left(\prod_{\rho \in \partial \Delta \cap M} x_\rho^{\alpha(\rho) + 1}\right)
\]
where $a_\alpha$ are complex numbers. In particular, since the origin is contained in $\Delta^\circ \cap N$, $s_\infty = \prod_{\rho \in \partial \Delta \cap M} x_\rho$, whose vanishing locus is $\cup_{\rho \in \partial \Delta \cap M} D_\rho$, is a global section of $\omega^{-1}_{\hat{X}_\Delta}$. The divisors $D_\rho$ are called the {\em toric boundary divisors} of $\hat{X}_\Delta$. In terms of homogeneous coordinates, $D_\rho$ is the vanishing locus of $x_\rho$.

\subsection{Mirrors of toric weak Fano threefolds}\label{sect:conslg}

We will now construct a family of LG models for each $\Delta$, which we expect to be mirror to $\hat{X}_\Delta$. The mirror of a toric variety $\hat{X}_\Delta$ is generally expected to be the LG model $(\mathbb{C}^{\times \dim \Delta}, w)$ where $w$ is some Laurent polynomial whose Newton polytope is $\Delta$.  We will call this LG model {\em Givental's LG model}. A Laurent polynomial produces a map which is very far from being proper, so we cannot apply Theorem \ref{thm:3fold} to calculate the KKP Hodge numbers of Givental's Landau-Ginzburg model. The procedure described in the next paragraph constructs a partial compactification of $(\mathbb{C}^{\times \dim \Delta},w)$ which admits a type III compactification.
\begin{remark}
As described in Section \ref{ex:three}, mirror symmetry for a Fano manifold $X$ implicitly requires one to choose an anticanonical divisor $W$ in $X$. As we have already mentioned, if $W$ is smooth, then $w : Y \rightarrow \mathbb{C}$ is expected to be proper. Givental's LG mirror to a toric variety $\hat{X}_\Delta$ is the mirror obtained by taking $W$ to be the union over all toric boundary divisors in $\hat{X}_\Delta$. Smoothing the divisor $W$ should be though of as inserting horizontal divisors into $Y$. In effect, this is what the process described below should be doing.
\end{remark}

%Be more precise about where facts come from.

In this paragraph, we define families of LG models which we expect to be mirror to $\hat{X}_{\Delta}$. The reader is referred to \cite{bat,bat-bo} for the facts in this paragraph. The polytope $\Delta^\circ$ is contained in $N_\mathbb{R} := N \otimes \mathbb{R}$ and is integral with respect to $N$. We may associate to $\Delta^\circ$ the toric variety $\hat{X}_{\Delta^\circ}$ in the same way we constructed $\hat{X}_\Delta$ above. We will let $s_\infty = \prod_{\rho \in \partial\Delta^\circ\cap N}x_\rho$, and we may choose $s$ to be a generic global section of $\omega_{\hat{X}_{\Delta^\circ}}^{-1}$. We may then produce a pencil $\mathscr{P}(r,t)$ of anticanonical hypersurfaces in $\hat{X}_{\Delta^\circ}$ written as $\{rs - ts_\infty = 0\}$ over $\mathbb{P}^1_{r,t}$. The base locus of this pencil is just the intersection of $V = \{s = 0\}$ with the union of all toric boundary divisors $D_\rho$ in $\hat{X}_{\Delta^\circ}$. By the assumption that $\Delta$ is reflexive, the line bundle $\omega_{\hat{X}_\Delta}^{-1}$ is semiample \cite[Proposition 2.4]{bat-bo}. Therefore, since $V$ is generic, it follows that for each $\rho$, $V \cap D_\rho$ is a smooth, reduced curve and the union of all curves $V \cap D_\rho$ is normal crossings. We may sequentially blow up the curves in the base locus of $\mathscr{P}(r,t)$ to resolve indeterminacy. The result is a smooth variety $Z_\Delta$ which is fibered over $\mathbb{P}^1_{r,t}$. Call this map $f_\Delta$. Let $Y_\Delta := Z_\Delta \setminus f_\Delta^{-1}(\infty)$ and let $w_\Delta := f_\Delta|_{Y_\Delta}$.

\begin{proposition}\label{prop:isagoodcomp}
The pair $(Y_\Delta,w_\Delta)$ is a proper LG model and $(Z_\Delta,f_\Delta)$ is a type III compactification of $(Y_{\Delta},w_{\Delta})$.
\end{proposition}
\begin{proof}
Recall that if we have a smooth blow-up $\pi : \widetilde{X} \rightarrow X$ along a codimension 2 subvariety $C$ in $X$, and $E = \pi^{-1}(C)$, then $-K_{\widetilde{X}} = -\pi^*K_X - E$. Thus if $E_{C_1},\dots, E_{C_k}$ are the exceptional divisors of the sequential blow up map $\pi:Z_{\Delta} \rightarrow \hat{X}_{\Delta^\circ}$, it follows that
$$
-K_{Z_{\Delta}} = -\pi^* K_{\hat{X}_{\Delta^\circ}} - \sum_{i=1}^k E_{C_i}.
$$
which is just the class of any fiber of $f_\Delta$. The fiber of $f_\Delta$ over $\infty$ is the union of the proper transforms of $D_\rho \in \hat{X}_{\Delta^\circ}$. The refinement $\hat{\Sigma}(\Delta^\circ)$ of $\Sigma(\Delta^\circ)$ produces triangulation of $\partial \Delta^\circ$ whose 0-dimensional strata are $\partial \Delta^\circ \cap N$. The divisors $D_\rho \in \hat{X}_{\Delta^\circ}$ correspond to the vertices in this triangulation, the edges between two vertices $\rho_1$ and $\rho_2$ correspond to rational curves of intersection between $D_{\rho_1}$ and $D_{\rho_2}$, and the triangles with vertices $\rho_1,\rho_2$ and $\rho_3$ correspond to triple intersection points between $D_{\rho_1},D_{\rho_2}$ and $D_{\rho_3}$. Therefore, this triangulation of $\partial \Delta^\circ$ is the dual intersection complex of $\cup_{\rho \in \partial \Delta^\circ \cap N}D_\rho$, which is a sphere, since $\Delta^\circ$ is a convex polytope in $\mathbb{R}^3$. Blowing up at curves which meet the singularities of $\cup_{\rho \in \partial \Delta^\circ \cap N}D_\rho$ transversally does not affect this dual intersection complex. Since $K_{Z_\Delta}$ is $-F$ for $F$ the class of any fiber of $f_\Delta$, if we take $\mathscr{U}$ to be the preimage of a disc around infinity under $f_\Delta$, $K_{\mathscr{U}}$ is trivial. Hence the fiber at infinity of $f : Z_\Delta \rightarrow \mathbb{P}^1$ is a type III degeneration of K3 surfaces.

The threefold $\hat{X}_{\Delta^\circ}$ satisfies $h^{p,0}(\hat{X}_{\Delta^\circ}) = 0$ for $p \neq 0$, and blowing up at curves does not affect $h^{p,0}(\hat{X}_{\Delta^\circ})$, as these Hodge numbers are birational invariants. Therefore, $h^{p,0}(Z_\Delta) = 0$ for $ p \neq 0$. We may then conclude that $(Z_\Delta,f_\Delta)$ is a type III compactification of $(Y_\Delta,w_\Delta)$.
\end{proof}

\subsection{The KKP Hodge numbers of $(Y_\Delta,w_\Delta)$} 

As a consequence of Proposition \ref{prop:isagoodcomp}, we can apply Theorem \ref{thm:3fold} to compute the KKP Hodge numbers of the proper LG model $(Y_\Delta ,w_\Delta)$.  First we will establish some useful notation.

\begin{defn}
We let $\Delta[i]$ denote the collection of all faces of $\Delta$ of dimension $i$. If $F$ is a face of $\Delta$, we let $\ell(F)$ denote the number of integral points in $F$ and let $\ell^*(F)$ denote the number of integral points on the interior of $F$.
\end{defn}
Now we may compute the Hodge numbers of $Z_\Delta$.
\begin{proposition}\label{prop:torichn}
Let $Z_\Delta$ be as above. Then $h^{0,0}(Z_\Delta) = h^{3,3}(Z_\Delta) = 1$,
\begin{align*}
h^{1,2}(Z_\Delta) = h^{2,1}(Z_{\Delta}) & = \sum_{F \in \Delta[2]} \ell^*(F) \\
h^{1,1}(Z_\Delta) = h^{2,2}(Z_\Delta) & = 2\ell(\Delta^\circ) - 5 - \sum_{F \in \Delta^\circ[2]} \ell^*(F) + \sum_{F \in \Delta^\circ[1]} \ell^*(F) \ell^*(F^\circ)
\end{align*}
and $h^{i,j}(Z_\Delta) = 0$ for all other $i,j$.
\end{proposition}
\begin{proof}
On the big torus $(\mathbb{C}^\times)^3_{x,y,z}$ of $\hat{X}_{\Delta^\circ}$, there is a Laurent polynomial $f(x,y,z)$ which determines $V$ and so that the Newton polytope of $f(x,y,z)$ is $\Delta$. The restriction of $V$ to the big torus $(\mathbb{C}^\times)^2$ in any $D_v$ has Newton polytope is computed in the following way. Let $\rho \in \Delta^\circ$, and let $\Gamma(\rho)$ be the smallest face of $\Delta^\circ$ containing $\rho$ (here, $\rho$ need not be in $N$). The face $\Gamma(\rho)$ has a dual face $\Gamma(\rho)^\circ$ in $\Delta$ defined to be
$$
\Gamma(\rho)^\circ =  \{ \sigma \in \Delta : \langle \rho,\sigma \rangle = -1 \}.
$$
These faces satisfy $\dim \Gamma(\rho) + \dim \Gamma(\rho)^\circ = 2$. The restriction of $V$ to the big torus $(\mathbb{C}^\times)^2 \subseteq D_\rho$ has Newton polytope $\Gamma(\rho)^\circ$. Thus
\begin{enumerate}
\item If $\dim \Gamma(\rho)^\circ = 2$, then $D_\rho \cap V = \emptyset$,
\item If $\dim \Gamma(\rho)^\circ = 1$ then $D_\rho \cap V$ is a union of $1+\ell^*(\Gamma(\rho)^\circ)$ smooth rational curves.
\item If $\dim \Gamma(\rho)^\circ = 0$ then $D_\rho \cap V$ is a single smooth curve whose genus is $\ell^*(\Gamma(\rho))^\circ$ (this follows by \cite{dan-ko}).
\end{enumerate}
These statements follow from \cite[Theorem 2.5]{bat-bo} or an easy computation. Now, recall that if we let $\widetilde{X}$ be the blow up of a threefold $X$ in a smooth irreducible curve of genus $g$ then
\begin{align*}
h^{2,2}(\widetilde{X}) =h^{1,1}(\widetilde{X}) &=h^{1,1}(X) + 1\\
h^{2,1}(\widetilde{X}) =h^{1,2}(\widetilde{X}) & =h^{1,2}(X) + g
\end{align*}
see e.g. \cite[\S 7.3.3]{vois2}. Therefore $h^{2,1}(Z_\Delta)$ is the sum of the genera of the curves which were blown up. For each facet $F$ of $\Delta$, the dual face $F^\circ$ in $\Delta^\circ$ is a single vertex. Therefore, the sum of genera of the blown up curves is simply the number of points on the interior of facets of $\Delta$. This gives
\begin{align*}
h^{2,1}(Z_\Delta) &= \sum_{F \in \Delta[2]} \ell^*(F).
\end{align*}
To show that
\begin{equation}\label{eq:3}
h^{1,1}(Z_\Delta) = 2\ell(\Delta^\circ) - 5 - \sum_{F \in \Delta^\circ[2]} \ell^*(F) + \sum_{F \in \Delta^\circ[1]} \ell^*(F) \ell^*(F^\circ)
\end{equation}
first note that $h^{1,1}(\hat{X}_{\Delta^\circ}) = \ell(\Delta^\circ) -4$, then we count the number of times we have blown up $\hat{X}_{\Delta^\circ}$ to get $Z_\Delta$ (which is the number of irreducible curves in the base locus of the pencil $\mathscr{P}(r,t)$) and add the resulting numbers. For each integral point on the interior of an edge of $\Delta^\circ$, we add $\ell^*(F^\circ)$. This contributes a $\sum_{F\in \Delta^\circ[1]}\ell^*(F)(\ell^*(F^\circ)+1)$ term. Then we add 1 for every vertex of $\Delta^\circ$. These two terms add up to
\begin{equation}\label{eq:h11}
\sum_{F\in \Delta^\circ[1]}\ell^*(F)(\ell^*(F^\circ)+1) + \sum_{F \in \Delta^\circ[0]} \ell(F) = \ell(\Delta^\circ) - 1 - \sum_{F \in \Delta^\circ[2]}\ell^*(F) + \sum_{F\in \Delta^\circ[1]}\ell^*(F) \ell^*(F^\circ).
\end{equation}
At this point, (\ref{eq:3}) can be deduced by adding the right hand side of (\ref{eq:h11}) to $\ell(\Delta^\circ) -4$.
\end{proof}

We will now compute the KKP Hodge numbers of $(Y_\Delta,w_\Delta)$. First, we compute the dimension of primitive cohomology of a fiber $V$ of $w_\Delta$. A generic member $V$ of the pencil $\mathscr{P}(r,t)$ is biholomorphic to an anticanonical hypersurface $\{s=0\}$ in $\hat{X}_{\Delta^\circ}$, so we will use the notation $V$ to denote both. The following result is a modification of \cite[Proposition 4.4.2]{bat}. 
\begin{proposition}\label{prop:bat}
If $\Delta^\circ$ is a reflexive polytope of dimension 3 and $V$ is a generic anticanonical hypersurface in $\hat{X}_{\Delta^\circ}$, then the sublattice of $\Pic(V)$ spanned by the union of all irreducible curves in $D_v \cap V$ as $v$ runs over all elements of $\partial \Delta^\circ \cap N$ is of rank
$$
\ell(\Delta^\circ)-4 - \sum_{F \in \Delta^\circ[2]} \ell^*(F) + \sum_{F \in \Delta^\circ[1]} \ell^*(F) \ell^*(F^\circ).
$$
\end{proposition}
\begin{proof}
Let $\partial \hat{X}_{\Delta^\circ}$ denote $ \cup_{\rho \in \partial \Delta^\circ \cap N} D_\rho$, and let $\partial V = \partial \hat{X}_{\Delta^\circ} \cap V$. By the discussion above, $\partial V$ is a normal crossings union of the irreducible components of the curves $C_v = D_v \cap V$ as $v$ runs over all elements of $\partial \Delta^\circ \cap N$. We have a long exact sequence of cohomology groups,
\[
\dots \longrightarrow H^2(V) \longrightarrow H^2(\partial V) \longrightarrow H^3_c(V \setminus \partial V) \longrightarrow 0
\]
coming from the long exact sequence for compactly supported cohomology. By the toric Lefschetz hyperplane theorem \cite[Theorem 3.7]{dk}, we have that $H^3(V \setminus \partial V) \cong \mathbb{C}^3$, therefore the image of $H^2(V) \rightarrow H^2(\partial V)$ is of codimension 3. The map $H^2(V) \rightarrow H^2(\partial V)$ is pullback in cohomology. 

We may compose this with the pullback along the normalization map $n : \widetilde{\partial V} \rightarrow \partial V$ to get a map $H^2(V) \rightarrow H^2(\widetilde{\partial V})$, which also has image of codimension 3 by Corollary \ref{cor:ncomp}. Identify $\widetilde{\partial V}$ with $\sqcup C_i$ for $C_i$ its irreducible components. Then the map $H^2(V) \rightarrow H^2(\sqcup C_i) \cong \bigoplus_i H^2(C_i)$ is dual to the Gysin map
\[
\bigoplus_i H^0(C_i) \longrightarrow H^2(V)
\]
which has kernel of rank 3 and image the span of the classes of the curves $C_i$. Therefore, we have that the curves $C_i$ span a subspace of $H^2(V)$ of rank $3$ minus the number of curves $C_i$. This number was computed in the proof of Proposition \ref{prop:torichn} to be 
\[
\ell(\Delta^\circ) - 1 - \sum_{F \in \Delta^\circ[2]} \ell^*(F) + \sum_{F \in \Delta^\circ[1]} \ell^*(F) \ell^*(F^\circ).
\]
This proves the result.
\end{proof}
\begin{proposition}\label{prop:primitive}
 Let $ph(Y_\Delta,w_\Delta)$ be as in Definition \ref{def:ph}. Then
$$
ph(Y_\Delta,w_\Delta) = 26- \ell(\Delta^\circ) + \sum_{F \in \Delta^\circ[2]} \ell^*(F) - \sum_{F \in \Delta^\circ[1]} \ell^*(F) \ell^*(F^\circ).
$$
\end{proposition}
\begin{proof}
First, let $\pi : Z_\Delta \rightarrow \hat{X}_{\Delta^\circ}$ be the contraction map. We note that $H^2(Z_\Delta)$ is spanned by the pullbacks of a collection of generators of $H^2(\hat{X}_{\Delta^\circ})$ (which we may take to be the first Chern classes of the divisors $\{ D_\rho \}_{\rho \in \partial \Delta^\circ \cap N}$) and the exceptional divisors of $\pi$, which we will denote $\{E_{C_i}\}_{i=1}^h$, where $C_i$ is the curve $\pi(E_{C_i})$ in $V \cap \cup_{\rho \in \partial \Delta^\circ \cap N}$. Since the proper transforms of $D_\rho$ is contained in $f_\Delta^{-1}(\infty)$, the pullback of $D_\rho$ to a smooth fiber of $f_\Delta$ is 0. The smooth fiber of $f_\Delta$ over $0$ is the proper transform of $V$, so we will also use the notation $V$ to denote a smooth fiber of $f_\Delta$. The intersection of $E_{C_i}$ with $V$ is then identified with the curve in the base locus that $E_j$ contracts to. Therefore, the pullback of $E_{C_i}$ to $V$ is the class of $C_i$ in $H^2(V)$. Thus
\[
\mathrm{im}(H^2(Z_\Delta) \longrightarrow H^2(V))
\]
is the subspace spanned by union of all curves $D_\rho \cap V$, where $v$ runs over $\partial \Delta^\circ \cap N$. By Proposition \ref{prop:bat}, this has dimension
$$
\ell(\Delta^\circ)-4 - \sum_{F \in \Delta^\circ[2]} \ell^*(F) + \sum_{F \in \Delta^\circ[1]} \ell^*(F) \ell^*(F^\circ).
$$
According to the global invariant cycles theorem, the image of $H^2(Y_\Delta) \rightarrow H^2(V)$ is the same as the image of $H^2(Z_\Delta) \rightarrow H^2(V)$. This proves the proposition when we combine it with the fact that $H^2(V)$ has dimension 22.
\end{proof}

\begin{proposition}\label{lemma:rohsiepe}
\begin{align*}
f^{1,2}(Y_\Delta,w_\Delta) &= f^{2,1}(Y_\Delta,w_\Delta) \\ &= 24-\ell(\Delta^\circ) +\sum_{F \in \Delta^\circ[2]} \ell^*(F) - \sum_{F \in \Delta^\circ[1]} \ell^*(F) \ell^*(F^\circ) +  \sum_{F \in \Delta[2]} \ell^*(F).
\end{align*}
\end{proposition}
\begin{proof}
Since $(Y_\Delta,w_\Delta)$ admits a type III compactification $(Z_\Delta,f_\Delta)$, we may apply Theorem \ref{thm:3fold} to see that
\begin{equation}\label{recall}
f^{1,2}(Y_\Delta,w_\Delta) = f^{2,1}(Y_\Delta,w_\Delta) = ph(Y_\Delta,w_\Delta)-2 + h^{2,1}(Z_\Delta).
\end{equation}
Combining Propositions \ref{prop:torichn} and \ref{prop:primitive} with (\ref{recall}) we obtain the proposition.
\end{proof}

Now, since $(Y_\Delta,w_\Delta)$ admits a type III compactification $(Z_\Delta,f_\Delta)$, Theorem \ref{thm:3fold} implies that the only KKP Hodge numbers that we do not yet know are $f^{1,1}(Y_\Delta,w_\Delta)$ and $f^{2,2}(Y_\Delta,w_\Delta)$.
\begin{proposition}\label{prop:f11}
$$
f^{1,1}(Y_\Delta,w_\Delta) = f^{2,2}(Y_\Delta,w_\Delta) = 0.
$$
\end{proposition}
\begin{proof}
By Theorem \ref{thm:3fold}, it is enough to show that $h^2(Y_\Delta,V) = 0$. Our first goal will be to show that 
\begin{equation}\label{eq:H4}
\dim H^2(Y_\Delta) = \dim H_c^4(Y_\Delta) = \ell(\Delta^\circ)-4 - \sum_{F \in \Delta^\circ[2]} \ell^*(F) + \sum_{F \in \Delta^\circ[1]} \ell^*(F) \ell^*(F^\circ).
\end{equation}
By Proposition \ref{prop:lowhn}, we know that $H^1(Y_\Delta, V) = 0$. We know that $H^5(Y_\Delta,V)$ sits in the long exact sequence in cohomology,
\[
\dots \longrightarrow H^1(Y,V) \longrightarrow H^1(Y_\Delta) \longrightarrow H^1(V) \longrightarrow \dots
\]
hence $H^1(Y_\Delta) = 0$, therefore $H^5_c(Y_\Delta) = 0$. We have a long exact sequence,
\[
\dots \longrightarrow H^3(V) \cong 0 \longrightarrow H^4_c(Y_\Delta) \longrightarrow H^4(Z_\Delta) \longrightarrow H^4(D_\infty) \longrightarrow H^5_c(Y_\Delta) \cong 0 \longrightarrow \dots.
\] 
We know from Proposition \ref{prop:torichn} that $H^4(Z_\Delta)$ has dimension 
\begin{equation}\label{eq:H4z}
2\ell(\Delta^\circ) - 5 - \sum_{F \in \Delta^\circ[2]} \ell^*(F) + \sum_{F \in \Delta^\circ[1]} \ell^*(F) \ell^*(F^\circ).
\end{equation}
The fiber $D_\infty$ is normal crossings and has $\ell(\Delta) - 1$ components, therefore by Corollary \ref{cor:ncomp}, $h^4(D_\infty) = \ell(\Delta) - 1$. Thus (\ref{eq:H4}) follows from (\ref{eq:H4z}) and Poincar\'e duality. We must now compute $H^2(Y_\Delta,V)$. We have the long exact sequence,
\begin{equation}\label{eq:relt}
\dots \rightarrow H^1(V) \cong 0 \rightarrow H^2(Y_\Delta,V) \rightarrow H^2(Y_\Delta) \rightarrow H^2(V) \rightarrow \dots.
\end{equation}
We showed in Proposition \ref{prop:primitive} that $H^2(Y_\Delta) \rightarrow H^2(V)$ has image of dimension
\begin{equation}\label{eq:l3}
\ell(\Delta^\circ)-4 - \sum_{F \in \Delta^\circ[2]} \ell^*(F) + \sum_{F \in \Delta^\circ[1]} \ell^*(F) \ell^*(F^\circ).
\end{equation}
Therefore $\dim H^2(Y_\Delta,V)=0$ by combining (\ref{eq:l3}), (\ref{eq:relt}) and (\ref{eq:H4}).
\end{proof}

\subsection{Topological mirror symmetry} 

In this subsection, we will show that $\hat{X}_\Delta$ and $(Y_\Delta,w_\Delta)$ are topologically mirror dual. We have now computed both $f^{p,q}(Y_\Delta,w_\Delta)$ and $h^{p,q}(\hat{X}_\Delta)$. We know that 
\[
h^{0,0}(\hat{X}_\Delta) = h^{3,3}(\hat{X}_\Delta) = f^{3,0}(Y_\Delta,w_\Delta) = f^{0,3}(Y_\Delta,w_\Delta) = 1
\]
by Theorem \ref{thm:3fold}. We know that if $p+q \neq 3$, then $f^{p,q}(Y_\Delta,w_\Delta) = 0$ by Theorem \ref{thm:3fold} and Proposition \ref{prop:f11}. By Proposition \ref{prop:Xdhn}, we have that
\[
h^{1,1}(\hat{X}_\Delta) = h^{2,2}(\hat{X}_\Delta) = \ell(\Delta) - 4,
\]
and by Proposition \ref{lemma:rohsiepe}, we have that
\begin{align*}
f^{1,2}(Y_\Delta,w_\Delta) &= f^{2,1}(Y_\Delta,w_\Delta) \\ &= 24-\ell(\Delta^\circ) +\sum_{F \in \Delta^\circ[2]} \ell^*(F) - \sum_{F \in \Delta^\circ[1]} \ell^*(F) \ell^*(F^\circ) +  \sum_{F \in \Delta[2]} \ell^*(F).
\end{align*}
Therefore, in order to prove that Conjecture \ref{conj:KKP} holds for the pair $\hat{X}_\Delta$ and $(Y_\Delta, w_\Delta)$ we must show that
\begin{equation}\label{bigcobid}
\ell(\Delta) - 4 = 24-\ell(\Delta^\circ) +\sum_{F \in \Delta^\circ[2]} \ell^*(F) - \sum_{F \in \Delta^\circ[1]} \ell^*(F) \ell^*(F^\circ) +  \sum_{F \in \Delta[2]} \ell^*(F).
\end{equation}
This turns out to be a combinatorial fact about reflexive 3 dimensional polytopes and their polar duals. We were unable to find (\ref{bigcobid}) in the literature, however, it can be deduced from the following fact.
\begin{lemma}[{\cite[Theorem 5.1.16]{nillhaase}}]\label{combid}
If $\Delta$ is a reflexive polytope in dimension 3 then
\begin{equation}\label{cobideq}
24 = \sum_{F \in \Delta[1]} (\ell(F)-1) (\ell(F^\circ)-1).
\end{equation}\end{lemma}

Haase, Nill, and Paffenholz \cite{nillhaase} give a purely combinatorial proof of Lemma \ref{combid}. Work of Batyrev and Dais \cite[Corollary 7.10]{bd} can be used to give another proof of Lemma \ref{combid} in which the number 24 explicitly appears as the topological Euler characteristic of a K3 surface.
\begin{proposition}\label{prop:bigcobid}
If $\Delta$ is a reflexive polytope of dimension 3, then (\ref{bigcobid}) is equivalent to (\ref{cobideq}).
\end{proposition}
\begin{proof}
Note that $\ell(F) - 1 = \ell^*(F) + 1$, and expand (\ref{cobideq}) to get
\begin{equation}\label{eq:dias}
24 = \sum_{F \in \Delta[1]}(\ell^*(F) \ell^*(F) + \ell^*(F) +  \ell^*(F^\circ)  +1). 
\end{equation}
The facets, edges and vertices of $\Delta$ and $\Delta^\circ$ form polyhedral decompositions of $S^2$. We let $V_\Delta, E_\Delta$ and $F_\Delta$ denote the number of vertices edges and faces of the decomposition o $S^2$ associated to $\Delta$, and similarly, let $V_{\Delta^\circ}, E_{\Delta^\circ}$ and $F_{\Delta^\circ}$ be the number of vertices, edges and faces associated to the decomposition of $S^2$ coming from $\Delta^\circ$. By the fact that $\Delta$ and $\Delta^\circ$ are dual polytopes, we have:
\begin{equation}\label{eq:euler}
E_{\Delta} = E_{\Delta^\circ}, \quad V_{\Delta} = F_{\Delta^\circ}, \quad F_{\Delta} = V_{\Delta^\circ}.
\end{equation}
Observe that (\ref{eq:dias}) is equivalent to
\begin{equation}\label{reareul}
24 = \sum_{F \in \Delta[1]}\ell^*(F) \ell^*(F^\circ) + \sum_{F \in \Delta[1]}\ell^*(F) + \sum_{F^\circ \in \Delta^\circ[1]} \ell^*(F^\circ) + E_\Delta.
\end{equation}
By Euler's formula and (\ref{eq:euler}), we have that $E_\Delta = V_\Delta + V_{\Delta^\circ}  - 2$. Therefore (\ref{reareul}) is equivalent to
\begin{equation}\label{almostthere}
24 = \sum_{F \in \Delta[1]}\ell^*(F) \ell^*(F^\circ) + \sum_{F \in \Delta[1]}\ell^*(F) + \sum_{F^\circ \in \Delta^\circ[1]} \ell^*(F^\circ) + V_\Delta + V_{\Delta^\circ} - 2.
\end{equation}
The number of points in $\partial \Delta \cap M$ not contained in the interior of a facet of $\Delta$ is $V_\Delta + \sum_{F\in \Delta[1]}\ell^*(F)$, and similarly, $V_\Delta + \sum_{F^\circ \in \Delta^\circ[1]}\ell^*(F^\circ)$ is the number of points in $\partial \Delta^\circ \cap N$ not contained in the interior of a facet of $\Delta^\circ$. Thus
\begin{align*}
V_\Delta + \sum_{F \in \Delta[1]}\ell^*(F) &= \ell(\Delta) - 1 - \sum_{F \in \Delta[2]}\ell^*(F) \\
V_{\Delta^\circ} + \sum_{F^\circ \in \Delta^\circ[1]}\ell^*(F^\circ) &= \ell(\Delta^\circ) - 1 - \sum_{F^\circ \in \Delta^\circ[2]}\ell^*(F^\circ).
\end{align*}
Substituting these equations into (\ref{almostthere}), we get 
\[
24 = \sum_{F\in \Delta[1]} \ell^*(F) \ell^*(F^\circ) + \ell(\Delta) + \ell(\Delta^*) - \sum_{F\in \Delta[2]}\ell^*(F) - \sum_{F^\circ \in \Delta^\circ[2]}\ell^*(F^\circ) - 4
\]
which can be rearranged to produce (\ref{bigcobid}).
\end{proof}
Comparing Proposition \ref{prop:bigcobid} to the formula for $h^{1,1}(\hat{X}_\Delta)$ in Proposition \ref{prop:Xdhn} we obtain the following result.
\begin{corollary}
$$
h^{2,2}(\hat{X}_\Delta) = h^{1,1}(\hat{X}_\Delta) =f^{2,1}(Y_\Delta,w_\Delta) = f^{1,2}(Y_\Delta,w_\Delta).
$$
\end{corollary}

%\begin{remark}\label{rmk:ks}
%As noted by Kreuzer and Skarke \cite[pp. 8]{ks}, there is a fundamental problem with trying to verify lattice polarized mirror symmetry between Batyrev dual K3 hypersurfaces in toric varieties, which is that if $S \subseteq \hat{X}_\Delta$ and $V \subseteq \hat{X}_{\Delta^\circ}$ are generic anticanonical hypersufaces, then
%$$
%\rank \, \Pic(S) + \rank \, \Pic(V) = 20 + \sum_{F \in \Delta[1]} \ell^*(F)\ell^*(F^\circ)
%$$
%while lattice polarized mirror symmetry claims (\cite{dolgachev}) that we should have
%$$
%\rank \, \Pic(S) + \rank \, \Pic(V) = 20.
%$$
%This problem disappears when we look at mirror symmetry for smooth toric varieties, as we have just seen. The difference between $\Pic(\hat{X}_\Delta)$ and $\Pic(S)$ is compensated for by $h^{2,1}(Z_\Delta)$ in the mirror. This seems to suggest that Batyrev-Borisov mirror symmetry is the result of a more natural duality between a smooth toric variety and its Landau-Ginzburg mirror.
%\qed\end{remark}

\begin{theorem}\label{thm:weakfanohn}
Let $\Delta$ be a reflexive 3-dimensional polytope, let $\hat{X}_\Delta$ be a weak Fano toric threefold constructed as in Section \ref{sect:constor} and let $(Y_\Delta,w_\Delta)$ be a LG model constructed as in Section \ref{sect:conslg}. Then
\[
f^{3-p,q}(Y_\Delta,w_\Delta) = h^{p,q}(\hat{X}_\Delta).
\]
In other words, $\hat{X}_\Delta$ and $(Y_\Delta,w_\Delta)$ are topologically mirror dual.
\end{theorem}
\begin{remark}
One may apply results of Batyrev \cite{bat2} to give a different proof of Theorem \ref{thm:weakfanohn}, which works in all dimensions (see \cite[Chapter 2]{harthe}). The computations in this section are included as an illustration of how one may use Theorem \ref{thm:3fold} to prove Conjecture \ref{conj:KKP} in concrete situations. Furthermore, the computations in this section will be used to verify certain Hodge number predictions regarding mirror pairs of log Calabi--Yau varieties in forthcoming work of the author in collaboration with Katzarkov and Przyjalkowski. 
\end{remark}
\bibliographystyle{alpha}

\bibliography{LGHN-BIB}

\begin{thebibliography}{BHPVdV04}

\bibitem[AKO08]{ako2}
D.~Auroux, L.~Katzarkov, and D.~Orlov.
\newblock Mirror symmetry for weighted projective planes and their
  noncommutative deformations.
\newblock {\em Ann. of Math.}, pages 867--943, 2008.

\bibitem[Aur07]{aur1}
D.~Auroux.
\newblock Mirror symmetry and {$T$}-duality in the complement of an
  anticanonical divisor.
\newblock {\em J. G\"{o}kova Geom. Topol. GGT}, {1}:51--91, 2007.

\bibitem[Aur08]{aur2}
D.~Auroux.
\newblock Special {L}agrangian fibrations, wall-crossing, and mirror symmetry.
\newblock {\em Surv. Differ. Geom}, { 13}({\bf 1}):1--48, 2008.

\bibitem[Bat93]{bat2}
V.~Batyrev.
\newblock Variations of the mixed {H}odge structure of affine hypersurfaces in
  algebraic tori.
\newblock {\em Duke Math. J.}, 69({\bf 2}):349--409, 1993.

\bibitem[Bat94]{bat}
V.~Batyrev.
\newblock Dual polyhedra and mirror symmetry for {C}alabi-{Y}au hypersurfaces
  in toric varieties.
\newblock {\em J. Algebraic Geom.}, 3({\bf 3}):493--535, 1994.

\bibitem[BB96a]{bor}
V.~Batyrev and L.~Borisov.
\newblock Mirror duality and string-theoretic {H}odge numbers.
\newblock {\em Invent. Math.}, 126({\bf 1}):183--203, 1996.

\bibitem[BB96b]{bat-bo}
V.~Batyrev and L.~Borisov.
\newblock On {C}alabi-{Y}au complete intersections in toric varieties.
\newblock In {\em Higher-dimensional complex varieties ({T}rento, 1994)}, pages
  39--65. de Gruyter, Berlin, 1996.

\bibitem[BD96]{bd}
V.~Batyrev and D.~Dais.
\newblock Strong {McKay} correspondence, string-theoretic {H}odge numbers and
  mirror symmetry.
\newblock {\em Topology}, 35({\bf 4}):901--929, 1996.

\bibitem[BHPVdV04]{bpv}
W.~Barth, K.~Hulek, C.~Peters, and A.~Van~de Ven.
\newblock {\em Compact complex surfaces}, volume~4 of {\em Ergebnisse der
  Mathematik und ihrer Grenzgebiete. 3. Folge.}
\newblock Springer-Verlag, Berlin, second edition, 2004.

\bibitem[Cle77]{clem}
C.H. Clemens.
\newblock Degeneration of {K}{\"a}hler manifolds.
\newblock {\em Duke Math. J.}, 44({\bf 2}):215--290, 1977.

\bibitem[CLS11]{cls}
D.~Cox, J.~Little, and H.~Schenck.
\newblock {\em Toric varieties}, volume 124 of {\em Graduate Studies in
  Mathematics}.
\newblock American Mathematical Society, Providence, RI, 2011.

\bibitem[CP18]{cp}
I.~Cheltsov and V.~Przyjalkowski.
\newblock Katzarkov-{K}ontsevich-{P}antev conjecture for {F}ano threefolds.
\newblock {\em arXiv preprint {\sf arXiv:1809.09218}}, 2018.

\bibitem[Del71]{del-th1}
P.~Deligne.
\newblock Th\'{e}orie de {H}odge. {II}.
\newblock {\em Inst. Hautes \'{E}tudes Sci. Publ. Math.}, ({\bf 40}):5--57,
  1971.

\bibitem[Del74]{del-th2}
P.~Deligne.
\newblock Th\'{e}orie de {H}odge. {III}.
\newblock {\em Inst. Hautes \'{E}tudes Sci. Publ. Math.}, ({\bf 44}):5--77,
  1974.

\bibitem[DHNT16]{dht}
C.F. Doran, A.~Harder, A.~Y. Novoseltsev, and A.~Thompson.
\newblock Calabi-{Y}au threefolds fibred by mirror quartic {K}3 surfaces.
\newblock {\em Adv. Math.}, 298:369--392, 2016.

\bibitem[Dim04]{dim}
A.~Dimca.
\newblock {\em Sheaves in topology}.
\newblock Springer, 2004.

\bibitem[DK86]{dk}
V.I. Danilov and A.G. Khovanski\u{\i}.
\newblock Newton polyhedra and an algorithm for calculating {H}odge-{D}eligne
  numbers.
\newblock {\em Izv. Akad. Nauk SSSR Ser. Mat.}, 50({\bf 5}):925--945, 1986.

\bibitem[DK87]{dan-ko}
V.~Danilov and A.~Khovanski{\u\i}.
\newblock Newton polyhedra and an algorithm for computing {H}odge-{D}eligne
  numbers.
\newblock {\em Mathematics of the USSR-Izvestiya}, 29(2):279, 1987.

\bibitem[ESY17]{esy}
H.~Esnault, C.~Sabbah, and J.-D. Yu.
\newblock {$E_1$}-degeneration of the irregular {H}odge filtration.
\newblock {\em J. Reine Angew. Math.}, 729:171--227, 2017.
\newblock With an appendix by Morihiko Saito.

\bibitem[FS86]{fs}
R.~Friedman and F.~Scattone.
\newblock Type {III} degenerations of {K3} surfaces.
\newblock {\em Invent. Math.}, 83({\bf 1}):1--39, 1986.

\bibitem[Fuj80]{fujiki}
A.~Fujiki.
\newblock Duality of mixed {H}odge structures of algebraic varieties.
\newblock {\em Publ. Res. Inst. Math. Sci.}, 16(3):635--667, 1980.

\bibitem[Giv98]{giv}
A.~Givental.
\newblock A mirror theorem for toric complete intersections.
\newblock In {\em Topological field theory, primitive forms and related topics
  ({K}yoto, 1996)}, volume 160 of {\em Progr. Math.}, pages 141--175.
  Birkh\"{a}user Boston, Boston, MA, 1998.

\bibitem[GKR12]{gkr2}
M.~Gross, L.~Katzarkov, and H.~Ruddat.
\newblock Towards mirror symmetry for varieties of general type (preprint
  version).
\newblock {\em preprint {\tt arXiv:1202.4042v2[math.AG]}}, 2012.

\bibitem[GKR17]{gkr}
M.~Gross, L.~Katzarkov, and H.~Ruddat.
\newblock Towards mirror symmetry for varieties of general type.
\newblock {\em Adv. Math.}, 308:208--275, 2017.

\bibitem[GS73]{cs}
P.~Griffiths and W.~Schmid.
\newblock Recent developments in {H}odge theory.
\newblock {\em Discrete subgroups of Lie groups}, pages 31--127, 1973.

\bibitem[GS11]{gs}
M.~Gross and B.~Siebert.
\newblock From real affine geometry to complex geometry.
\newblock {\em Ann. of Math. (2)}, 174({\bf 3}):1301--1428, 2011.

\bibitem[Har16]{harthe}
A.~Harder.
\newblock {\em The geometry of {L}andau-{G}inzburg models}.
\newblock PhD thesis, University of Alberta, 2016.

\bibitem[Her03]{hert}
C.~Hertling.
\newblock {$tt^*$} geometry, {F}robenius manifolds, their connections, and the
  construction for singularities.
\newblock {\em J. Reine Angew. Math.}, 555:77--161, 2003.

\bibitem[HNP12]{nillhaase}
C.~Haase, B.~Nill, and A.~Paffenholz.
\newblock {\em Lecture Notes on Lattice Polytopes}.
\newblock 2012.

\bibitem[ILP13]{prz3}
N.~O. Ilten, J.~Lewis, and V.~Przyjalkowski.
\newblock Toric degenerations of {F}ano threefolds giving weak
  {L}andau-{G}inzburg models.
\newblock {\em J. Algebra}, 374:104--121, 2013.

\bibitem[KKP08]{kkp1}
L.~Katzarkov, M.~Kontsevich, and T.~Pantev.
\newblock Hodge theoretic aspects of mirror symmetry.
\newblock In {\em From {H}odge theory to integrability and {TQFT}
  tt*-geometry}, volume~78 of {\em Proc. Sympos. Pure Math.}, pages 87--174.
  Amer. Math. Soc., Providence, RI, 2008.

\bibitem[KKP17]{kkp2}
L.~Katzarkov, M.~Kontsevich, and T.~Pantev.
\newblock Bogomolov-{T}ian-{T}odorov theorems for {L}andau-{G}inzburg models.
\newblock {\em J. Differential Geom.}, 105({\bf 1}):55--117, 2017.

\bibitem[Kon95]{kont}
M.~Kontsevich.
\newblock Homological algebra of mirror symmetry.
\newblock In {\em Proceedings of the {I}nternational {C}ongress of
  {M}athematicians, {V}ol. 1, 2 ({Z}\"{u}rich, 1994)}, pages 120--139.
  Birkh\"{a}user, Basel, 1995.

\bibitem[Kul77]{kul}
V.~Kulikov.
\newblock Degenerations of {$K3$} surfaces and {E}nriques surfaces.
\newblock {\em Izv. Akad. Nauk SSSR Ser. Mat.}, 41({\bf 5}):1008--1042, 1199,
  1977.

\bibitem[LP18]{lp}
V.~Lunts and V.~Przyjalkowski.
\newblock Landau-{G}inzburg {H}odge numbers for mirrors of del {P}ezzo
  surfaces.
\newblock {\em Adv. Math.}, 329:189--216, 2018.

\bibitem[Mor]{mor}
D.~Morrison.
\newblock The {C}lemens-{S}chmid exact sequence and applications.
\newblock In {\em Topics in transcendental algebraic geometry}, volume 106 of
  {\em Ann. of Math. Stud.}, pages 101--119. Princeton Univ. Press, Princeton,
  NJ.

\bibitem[Pha85]{odl}
F.~Pham.
\newblock La descente des cols par les onglets de {L}efschetz, avec vues sur
  {G}auss-{M}anin.
\newblock {\em Ast\'{e}risque}, ({\bf 130}):11--47, 1985.
\newblock Differential systems and singularities (Luminy, 1983).

\bibitem[Prz18]{prz}
V.~Przjalkowski.
\newblock On the {C}alabi-{Y}au compactifications of toric {L}andau-{G}inzburg
  models for {F}ano complete intersections.
\newblock {\em Mat. Zametki}, 103({\bf 1}):111--119, 2018.

\bibitem[PS08]{pet-st}
C.~Peters and J.~Steenbrink.
\newblock {\em Mixed {H}odge structures}, volume~52 of {\em Ergebnisse der
  Mathematik und ihrer Grenzgebiete. 3. Folge.}
\newblock Springer-Verlag, Berlin, 2008.

\bibitem[PS15]{ps}
{V}. Przyjalkowski and {C}. Shramov.
\newblock On {H}odge numbers of complete intersections and {L}andau-{G}inzburg
  models.
\newblock {\em Int. Math. Res. Not.}, 2015({\bf 21}):11302--11332, 2015.

\bibitem[Sch73]{sch}
W.~Schmid.
\newblock Variation of {H}odge structure: the singularities of the period
  mapping.
\newblock {\em Invent. Math.}, 22({\bf 3}):211--319, 1973.

\bibitem[Sha18]{sham}
Y.~Shamoto.
\newblock Hodge-{T}ate conditions for {L}andau-{G}inzburg models.
\newblock {\em Publ. Res. Inst. Math. Sci.}, 54({\bf 3}):469--515, 2018.

\bibitem[SYZ96]{syz}
A.~Strominger, S.-T. Yau, and E.~Zaslow.
\newblock Mirror symmetry is {T}-duality.
\newblock {\em Nuclear Phys. B}, 479({\bf 1-2}):243--259, 1996.

\bibitem[Voi07a]{vois2}
C.~Voisin.
\newblock {\em Hodge theory and complex algebraic geometry. {I}}, volume~76 of
  {\em Cambridge Studies in Advanced Mathematics}.
\newblock Cambridge University Press, Cambridge, english edition, 2007.
\newblock Translated from the French by Leila Schneps.

\bibitem[Voi07b]{vois3}
C.~Voisin.
\newblock {\em Hodge theory and complex algebraic geometry. {II}}, volume~77 of
  {\em Cambridge Studies in Advanced Mathematics}.
\newblock Cambridge University Press, Cambridge, english edition, 2007.
\newblock Translated from the French by Leila Schneps.

\end{thebibliography}

\end{document}